\documentclass[11pt]{article}
\usepackage{amsmath}
\usepackage{amssymb}
\usepackage{amsthm}
\usepackage{latexsym}
\usepackage{color}
\usepackage{graphicx}
\usepackage{color}
\usepackage{caption}
\usepackage{subcaption}
\usepackage[normalem]{ulem}
\usepackage{enumitem}

 \usepackage{xcolor}
 
\DeclareSymbolFont{calletters}{OMS}{cmsy}{m}{n}
\DeclareSymbolFontAlphabet{\mathcal}{calletters}

\def\be{\begin{eqnarray}}
\def\ee{\end{eqnarray}}

\def\b*{\begin{eqnarray*}}
\def\e*{\end{eqnarray*}}

\newtheorem{Theorem}{Theorem}[section]

\newtheorem{Proposition}[Theorem]{Proposition}

\newtheorem{Assumption}[Theorem]{Assumption}

\newtheorem{Lemma}[Theorem]{Lemma}

\newtheorem{Remark}[Theorem]{Remark}
\newtheorem{Example}[Theorem]{Example}

\makeatletter \@addtoreset{equation}{section}

\usepackage{color}

\newcommand{\bea}{\begin{eqnarray}}
\newcommand{\bes}{\begin{subequations}}
\newcommand{\ees}{\end{subequations}}
\newcommand{\bgt}{\begin{gather}}
\newcommand{\egt}{\begin{gather}}
\newcommand{\eea}{\end{eqnarray}}
\newcommand{\beaa}{\begin{eqnarray*}}
\newcommand{\eeaa}{\end{eqnarray*}}

\def \E{\mathbb{E}}
\def \F{\mathbb{F}}

\def \P{\mathbb{P}}

\def \R{\mathbb{R}}

\def\Lc{{\cal L}}

\def\Pc{{\cal P}}

\def\Cb{\overline{C}}

\def\Eb{E}

\def \eps{\varepsilon}

\def \0{\mathbf{0}}

\def \Xt{\widetilde{X}}

\def \x{\times}

\def\mub{\bar{\mu}}
\def\sigmab{\bar{\sigma}}
\def\mut{\tilde{\mu}}
\def\sigmat{\tilde{\sigma}}
\def\Lct{\widetilde{\Lc}}
\def\ft{\tilde{f}}
\def\Phit{\widetilde{\Phi}}
\def\Deltat{\widetilde{\Delta}}
\def\xt{\tilde x}
\def\yt{\tilde y}

\def\1{\mathbf{1}}
\def\xr{{\rm x}}
\def\vr{{\rm v}}

 \def\Cb{{\mathbb C}}

  \def\wr{{\rm w}}
  \def\vs#1{\vspace{2mm}}

 \def\Tr{{\rm Tr}}

 \def\det#1{{\rm det}\left(#1\right)}
 \def\ar{\mathfrak{a}}
 \def\br{\mathfrak{b}} 
 
  \def\wr{{\rm w}}

 \def\AMp{{\rm \bf A}}
 \def\Ap{\overrightarrow{\rm A}}
\def\Sigmab{\Sigma}

\def\vr{{\rm v}}

 \def\fr{{\rm f}}

\title{On the regularity of solutions of some linear parabolic path-dependent PDEs}

\author{
Bruno Bouchard
\footnote{CEREMADE, Universit\'e Paris-Dauphine, PSL, CNRS.  bouchard@ceremade.dauphine.fr. } 
\and 
Xiaolu Tan
\footnote{Department of Mathematics, The Chinese University of Hong Kong. xiaolu.tan@cuhk.edu.hk. }
}

\begin{document}
\maketitle

\begin{abstract} 
	We study a class of linear parabolic path-dependent PDEs (PPDEs) defined on the space of c\`adl\`ag paths $\xr \in D([0,T])$, 
	in which the coefficient functions at time $t$ depend on $\xr(t)$ and $\int_{0}^{t}\xr(s)dA_{s}$, for some (deterministic) continuous function $A$ with bounded variations.
	Under uniform ellipticity and  H\"older regularity conditions on the coefficients, together with some technical conditions on $A$,
	we obtain the existence of a smooth solution to the PPDE by appealing to the notion of Dupire's derivatives. 
	It provides a generalization to the existing literature studying the case where $A_t = t$,
	and complements {our recent work in} \cite{bouchard2021approximate} on the regularity of approximate viscosity solutions for parabolic PPDEs.
	As a by-product, we also obtain   existence and uniqueness of weak solutions for a class of path-dependent SDEs.
\end{abstract}

\\
{\bf MSC2020 subject classifications:}  35K65,60H10.

\section{Introduction} 
 
	We consider   linear parabolic path-dependent PDEs (PPDEs) of the form 
	\begin{align}\label{eq:ppde_intro} 
		\partial_{t}\vr ~+~ \mub \partial_{\xr} \vr ~+~\frac12 \sigmab^{2} \partial^{2}_{\xr} \vr ~+~ \bar \ell
		&=
		0,
		~\mbox{on}~
		[0,T)\x D([0,T])\\
		\vr(T,\cdot)&=\bar g~\mbox{on}~
		  D([0,T]).\nonumber
	\end{align}
	In the above, $D([0,T])$ denotes the space of all  real-valued c\`adl\`ag path $\xr = (\xr(t))_{t \in [0,T]}$ on $[0,T]$, 
	the derivatives are taken in the sense of Dupire \cite{dupireito,cont2013functional} (see Section \ref{sec:conditions} below),
	and the coefficient functions $(\mub, \sigmab, \bar \ell,\bar g): [0,T] \x D([0,T]) \longrightarrow \R \x \R \x \R\x \R$ are of the form 
	$$
		\big( \mub_{t}, \sigmab_{t}, \bar \ell_t,\bar g \big) ( \xr)
		~=~
		\big( \mu_{t}, \sigma_{t}, \ell_{t}\big) \big( \xr(t), I_t(\xr) \big),\;\bar g(\xr)=g(\xr_{T},I_{T}(\xr)),
		~~\mbox{with}~
		I_t(\xr) := \int_{0}^{t} \xr(s)dA_{s},
	$$
	for some functions $(\mu, \sigma, \ell,g): [0,T] \x \R^2 \longrightarrow \R \x \R \x \R \x \R$, and a continuous process $A$ with bounded variations. 
	When $A$ is absolutely continuous, say simply $A_t = t$, the above can be written as a degenerate parabolic PDE 
	\begin{equation} \label{eq:pde_intro}
		\partial_{t} v +  \mu \partial_{ x_{1}} v+  x_{1} \partial_{x_{2}} v +\frac12  \sigma^{2} \partial^{2}_{x_{1}x_{1}} v + \ell = 0
		~\mbox{on}~
		[0,T)\x \R^{2},\;v(T,\cdot)=g ~\mbox{on}~
		  \R^{2},
	\end{equation}
	in which the derivatives are now taken in the usual sense and 
	$$
		\vr(t,\xr)
		~=~
		v \big( t,\xr_{t}, I_t(\xr) \big).
	$$
	Indeed, the Dupire's  horizontal derivative $\partial_{t}\vr$ and vertical derivatives $(\partial_{\xr} \vr,\partial^{2}_{\xr} \vr)$ are related to the partial derivatives of $v$ through
	$$
		(\partial_{t},\partial_{\xr} ,\partial^{2}_{\xr} ) \vr(t,\xr)
		=
		\left(\partial_{t}+\xr(t) \partial_{x_{2}},\partial_{x_{1}},\partial^{2}_{x_{1}x_{1}}\right)
		v \big(t,\xr(t), I_t(\xr) \big).
	$$
	{Various} works are devoted to such equations, going back to \cite{kolmogorov1934theorie}, in more complex multivariate frameworks, see e.g.~\cite{delarue2010density,francesco2005class,lanconelli2002linear,sonin1967class,weber1951fundamental} and the references therein. The latter PDE may not admit a $C^{1,2}$-solution, in the traditional sense, even when $\sigmab$ is uniformly elliptic: $\partial_{t}v$ and $ \partial_{x_{2}} v$ are in general not well-defined and one needs to define $\partial_{t}v + x_{1} \partial_{x_{2}} v$ jointly, appealing to the notion of Lie derivative, which amounts to considering Dupire's horizontal derivative when the PDE is seen as a  PPDE. 

	\vspace{0.5em}

	The main novelty of this paper is that we do not assume anymore that  $(A_t)_{t\in [0,T]}$ is absolutely continuous in $t$. 
	In this case, the PDE formulation \eqref{eq:pde_intro} is not valide anymore, but the PPDE formulation \eqref{eq:ppde_intro} is still adequate. 
	We provide conditions under which \eqref{eq:ppde_intro} admits a solution  that is smooth in the sense of Dupire's deviratives. It complements \cite{bouchard2021approximate} in which coefficients are assumed to be $C^{1+\alpha}$, which allows one to construct the so-called approximate viscosity solutions of non-linear path-dependent PDEs   with first order Dupire's vertical derivative enjoying some H\"older-type regularity 
	(see  \cite{bouchard2021approximate} for details). 
	As shown in e.g.~\cite{bouchard2021c}, in many situations, this is already sufficient to derive a Feynman-Kac's representation of the solution by appealing to a version of It\^{o}-Dupire's stochastic calculus for path-dependent functionals that are only vertically differentiable up to the first order.
	In contrast to \cite{bouchard2021approximate}, we only assume here that the coefficients are H\"older continuous, but require  $\bar \sigma$ to be non-degenerate, so as to expect the classical regularization effect to operate. 

	\vspace{0.5em}

	We rely on the parametrix approach, see e.g.~\cite[Chapter 1]{friedman2008partial}. For this, we perform a change of variables which allows us to reduce to a PDE of the form 
	$$
		\partial_{t} u +   \mu  \langle (1,A), D u \rangle + \frac12   \sigma^{2}  (1,A) D^{2} u (1,A)^{\top} = 0,
	$$
	which can be written even if $A$  is not absolutely continuous. 
	The above is again degenerate and $(D u,D^{2} u)$ may not be well-defined. However, the parametrix approach allows one to show that $ \langle (1,A), D u \rangle$ and $(1,A) D^{2}   u (1,A)^{\top}$ are, 
	which in turn implies that the vertical derivatives $(\partial_{\xr} \vr,\partial^{2}_{\xr} \vr)$ of the path-dependent functional $\vr$ are.

	\vspace{0.5em}

	As a by-product, we establish  the existence and uniqueness of a weak solution to the path-dependent stochastic differential equation (SDE)
 	$$
		X_t 
		=
		X_0
		+
		\int_0^t \mu_{s} ( X_s , I_s ) d s 
		+ 
		\int_0^t \sigma_{t} ( X_s , I_s ) dW_t,
		~~~
		I_t = \int_{0}^{t} X_s d A_s,
		~~
		t \ge 0,
	$$
	and provide 	some first properties of the transition density of the Markov process $( X, I)$,
	as well as the corresponding Feynman-Kac's formula.

	\vspace{0.5em}

{These results require structural conditions relating the H\"older regularity of the coefficient $(\mu,\sigma)$ and the path behavior of $A$. If one knows a priori that the above SDE admits a unique weak solution, then one can prove under weaker conditions  that the candidate solution to  \eqref{eq:ppde_intro}, deduced from a formal application of the Feynman-Kac's formula\footnote{Or more rigorously its viscosity solution in the sense of \cite{cosso2021path,zhou2020viscosity}, see also e.g.~\cite{bouchard2021approximate,ekren2014viscosity,ren2014comparison} and the references therein for an alternative definition.}, is already $C^{1}$ in space, in the sense of Dupire. As mentioned above, this turns out to be enough to deduce its It\^{o}-Dupire's semimartingale decomposition.}

	\vspace{0.5em}

	 All over this paper, we stick to a one-dimensional setting for ease of notations. Extensions to multivariate frameworks can be provided by using similar techniques.

	\vspace{0.5em}

	The rest of this paper is organized as follows. 
	Section \ref{sec: main results} states our main results. 
	Proofs are collected in Section \ref{sec: proofs}.

	\vspace{0.5em}
	
	In the following, the $i$-th component of a vector $x$ is denoted by $x_{i}$, the $(i,j)$-component of a matrix $M$ is denoted by $M_{ij}$. 
	Given $\phi: (t,x)\in [0,T] \x \R^2 \longrightarrow \phi(t,x)\in \R$, we let $D\phi$ and $D^{2}\phi$ (or $D_{x}\phi$ and $D^{2}_{xx}\phi$) be the gradient and the Hessian matrix with respect to $x$. 
	The space partial derivatives are denoted by $\partial_{x_{i}}\phi$,  $\partial^{2}_{x_{i}x_{j}}\phi$, and so on if we have to consider higher orders.

\section{Dupire's regularity for linear PPDEs depending on the average of the path} \label{sec: main results}

\subsection{Notations and assumptions}
\label{sec:conditions}

	Given $T> 0$, let $D([0,T])$ denote the Skorokhod space of all $\R$--valued c\`adl\`ag paths $\xr=(\xr(t))_{t \in [0,T]}$ on $[0,T]$,
	and let $C([0,T])$ denote  the subspace of continuous paths.
	Let us equipped $D([0,T])$ with the Skorokhod topology, and $C([0,T])$ with the uniform convergence topology.
	Let $A = (A_t)_{t \ge 0}$ be a deterministic continuous process with finite variation, 
	and $(\mu, \sigma, \ell): [0,T] \x \R^2 \longrightarrow \R \x \R \x \R$ be   coefficient functions,
	from which we define path-dependent functionals $(\mub, \sigmab, \bar \ell): [0,T] \x D([0,T]) \longrightarrow \R \x \R \x \R$ by
	$$
		\big( \mub_{t}, \sigmab_{t}, \bar \ell_t \big) ( \xr)
		~:=~
		\big( \mu_{t}, \sigma_{t}, \ell_{t} \big) \big( \xr(t), I_t(\xr) \big),
		~~\mbox{with}~
		I_t(\xr) := \int_{0}^{t} \xr(s)dA_{s}.
	$$
	We study the following linear parabolic path-dependent PDE (PPDE): 
	\begin{align}\label{eq:ppde} 
		\partial_{t}\vr ~+~ \mub \partial_{\xr} \vr ~+~\frac12 \sigmab^{2} \partial^{2}_{\xr} \vr ~+~ \bar \ell
		~=~
		0,
		~\mbox{on}~
		[0,T)\x D([0,T]),
	\end{align}
	with terminal condition $\vr(T, \xr) = \bar g(\xr):= g \big(\xr(T), I_T(\xr) \big)$ for some function $g: \R \x \R \longrightarrow \R$. In the above, the derivatives are taken in the sense of Dupire.

\paragraph{Dupire's derivatives for path-dependent functionals}

To give a precise definition to the PPDE \eqref{eq:ppde}, let us recall Dupire's \cite{dupireito,cont2013functional} notion of  horizontal derivative $\partial_t$ and  vertical derivatives $\partial_{\xr}$ and $\partial^2_{\xr}$ for path-dependent functionals.

	\vspace{0.5em}

	Let $F: [0,T] \x D([0,T]) \longrightarrow \R$ be a path-dependent functional,
	it is said to be non-anticipative if $F(s,\xr)=F(s,\xr(s\wedge \cdot))$ for all $(s,\xr)\in [0,T] \x D([0,T])$.
	For an non-anticipative map $F$, its horizontal derivative $\partial_s F(s, \xr)$  at $(s,\xr) \in [0,T) \x D([0,T])$ is defined as
	$$
		\partial_{s}F(s,\xr) ~:=~ \lim_{h\searrow 0} \frac{F(s+h,\xr(s \wedge\cdot )) - F(s,\xr)}{h},
	$$
	and its vertical derivative $\partial_{\xr} F(s, \xr)$ is defined as
	$$
		\partial_{\xr} F(s, \xr) ~:=~ \lim_{y \to 0} \frac{F(s,\xr+ y\1_{[s,T]}) - F(s, \xr)}{y},
	$$
	whenever the limits exist.
	In the above, $\xr+ y \1_{[s,T]}$ denotes the path taking value $\xr_t + y\1_{[s,T]}(t)$ at time $t \in [0,T]$.
	Similarly, one can define the second order vertical derivative $\partial^2_{\xr} F$ as the vertical derivative of $\partial_{\xr} F$.
	Given $t \in (0,T]$, we denote by $\Cb([0,t))$ the space of all continuous non-anticipative functionals $F: [0,t) \x D([0,T]) \longrightarrow \R$,
	and we set  
	$$
		\Cb^{0,1}([0,t)) 
		:= 
		\big\{ 
			F \in \Cb([0,t)) ~:  \partial_{\xr}F ~\mbox{is well-defined and belongs to}~
			 \Cb([0,t)) 
		\big\},
	$$
	as well as  
	$$
		\Cb^{1,2}([0,t))
		:=
		\big\{
			F \in \Cb^{0,1}([0,t))~:
			\partial_s F~\mbox{and}~\partial^{2}_{\xr}F~\mbox{are well-defined and  belong to}~\Cb([0,t))
		\big\}.
	$$

\paragraph{Assumptions on the process $A$:}

	Recall that $A = (A_t)_{t \in [0,T]}$ is a deterministic process with finite variation.
	For $0 \le s < t \le T$, let us define  $\overline A_{s,t}:=\frac1{t-s} \int_{s}^{t} A_{r}dr$ and
	$$
		m_{s,t}:=\frac1{t-s}\int_{s}^{t} \big(A_{r}-\overline A_{s,t} \big)^{2} dr,
		~~~
		\tilde m_{s,t}:=\frac1{t-s}\int_{s}^{t} (A_{r}-A_{s} )^{2} dr. 
	$$
	The above will play a major role in our analysis, as they will drive the behavior of the parametrix density on small time intervals.

	\begin{Assumption}\label{hyp:A}
		$\mathrm{(i)}$ There exist constants $\beta_0, \beta_1, \beta _2, \beta_3 \ge 0$ and $C_{(\ref{eq : hyp vitesse explo})}, C_{(\ref{eq:order_m})}> 0$ such that,
		for all $0\le s< t \le T$,
		\begin{align}
			\frac{1}{C_{(\ref{eq : hyp vitesse explo})}} (t-s)^{-\beta_1} 
			~\le~
			&\frac{\tilde m_{s,t}}{m_{s,t}} 
			~\le~
			C_{(\ref{eq : hyp vitesse explo})} (t-s)^{-\beta_0}, 
			\label{eq : hyp vitesse explo}\\
			\frac{1}{C_{(\ref{eq:order_m})}} (t-s)^{- \beta_2}
			~\le~ 
			&\frac{1}{m_{s,t}}
			~\le~
			C_{(\ref{eq:order_m})} (t-s)^{- \beta_3}.
			\label{eq:order_m}
		\end{align}
	
		\noindent $\mathrm{(ii)}$ There exist constants $\beta_4 \ge 0$ and
		$C_{\eqref{eq:holder A}}>0$ such that	
		\begin{align}
			|A_t-A_s| ~\le~ C_{\eqref{eq:holder A}} (t-s)^{\beta_4},
			~~\mbox{for all}~ 0 \le s< t \le T.
			\label{eq:holder A}
		\end{align}
	\end{Assumption}

	\begin{Remark}\label{rem : ordre beta} 
		Notice that, $\lim_{t\downarrow s}\tilde m_{s,t}=0$ by continuity of $A$. 
		Without loss of generality, one can therefore assume that 
		$$
			\beta_{1}\le \beta_{0}\le \beta_{3}\mbox{ and } \beta_{1}\le \beta_{2}\le \beta_{3}.
		$$
		Moreover, one can always choose $\beta_{0}=\beta_{4}=0$ since $m_{s,t} \le \tilde m_{s,t}$ by their definitions.
	\end{Remark}

	Let us provide some typical examples. 
	\begin{Example}\label{example : cas de A}
		$\mathrm{(i)}$ Let $A$ be defined by $A_{t}= \int_{0}^{t} \rho(s) ds$, $t\ge 0$, with $\eps\le \rho\le 1/\eps$ a.e.~for some $\eps>0$.	
		Then, it is easy to check that Assumption \ref{hyp:A} holds with
		$$
			\beta_0 = 0,
			~~\beta_1 = 0, 
			~~\beta_2 = 2,
			~~\beta_3 = 2,
			~~\beta_4 = 1.
		$$
		In this setting, our main results are   similar to those in \cite{francesco2005class}, which studied a multivariate version of the  case where $\rho$ is constant.

		\vspace{0.5em}
		
		\noindent $\mathrm{(ii)}$
		Let $A_{t}=t^{\gamma}$ for some $\gamma \in (0,1)$. Then, 
		$$
			m_{s,t} 
			~=~ 
			\frac{t^{2\gamma+1}-s^{2\gamma+1}}{(2\gamma+1)(t-s)}-\frac{|t^{\gamma+1}-s^{\gamma+1}|^{2}}{(\gamma+1)^{2}(t-s)^{2}},
		$$
		and
		$$
			\tilde m_{s,t}
			~=~
			\frac{\frac1{2\gamma+1}(t^{2\gamma+1}-s^{2\gamma+1})-2\frac1{\gamma+1}(t^{\gamma+1}-s^{\gamma+1})s^{\gamma}+(t-s)s^{2\gamma}}{t-s}.
		$$
		In this setting, Assumption \ref{hyp:A} holds true with
		$$
			\beta_0 = 0,
			~~\beta_1 = 0,
			~~\beta_2 = 2 \gamma, 
			~~\beta_3 = 2\gamma,
			~~\beta_4 = \gamma.
		$$
		
		\vspace{0.5em}
		
		\noindent $\mathrm{(iii)}$ Assume that  there exists $1\ge \gamma_{1}\ge \gamma_{2}>0$ and $C_{1},C_{2}>0$  such that 
		$$
			C_{1}|t-s|^{\gamma_{1}}\le A_{t}-A_{s}\le C_{2} |t-s|^{\gamma_{2}},\mbox{ for all } s\le t\le T. 
		$$
		Then,  Assumption \ref{hyp:A} holds with
		$$
			\beta_0 = 2(\gamma_{1}-\gamma_{2}),
			~~\beta_1 = 0,
			~~\beta_2 = 2 \gamma_{2}, 
			~~\beta_3 = 2\gamma_{1},
			~~\beta_4 = \gamma_{2}.
		$$
		Indeed,  let us choose $t_{0} \in [0,T]$ such that $A_{t_{0}}=\overline A_{s,t}$. 
		Assume that $t-t_{0}\ge (t-s)/2$ (otherwise  $t_{0}-s\ge (t-s)/2$ and we can use similar computations), then 
		\begin{align*} 
			{(t-s)m_{s,t}}
			~=~
			\int_{s}^{t}|A_{r}-A_{t_{0}}|^{2}
			&~\ge~
			\int_{t_{0}}^{t}C_{1}^{2}|{r}-{t_{0}}|^{2\gamma_{1}}dr
			~\ge~
			\frac{C_{1}^{2}}{2^{2\gamma_{1}+1}(2\gamma_{1}+1)}|t-s|^{2\gamma_{1}+1}.
		\end{align*}
		On the other hand, $ {(t-s)m_{s,t}\le } \int_{s}^{t}|A_{r}-A_{s}|^{2}\le \frac{C_{2}^{2}}{2\gamma_{2}+1} |t-s|^{2\gamma_{2}+1}$. Thus,  
		$$
			1
			~\le~
			\frac{\tilde m_{s,t}}{m_{s,t}}
			~\le~
			2^{2\gamma_{1}+1}\frac{C_{2}^{2}(2\gamma_{1}+1)}{C_{1}^{2}(2\gamma_{2}+1)}
			|t-s|^{-2(\gamma_{1}-\gamma_{2})}.
		$$
	 \end{Example}

\paragraph{Assumptions on the coefficient functions $\mu$ and $\sigma$:}

	As in e.g.~\cite{francesco2005class}, the following H\"older regularity assumption on the coefficient functions
	$(\mu, \sigma): [0,T] \x \R^2 \longrightarrow \R \x \R$
	is calibrated to match with the explosion rate of the quadratic form entering the parametrix.
	It does not impose smoothness conditions on $\mu$ and $\sigma$ as in e.g.~\cite{sonin1967class}, and facilitate the analysis, see Remark \ref{rem: si holder} below.
	Let us set 
	$$
		\Theta
		~:=~
		\big\{
			(s,x,t,y) \in [0,T]\x \R^{2}\x [0,T]\x \R^{2}~:~s< t 
		\big\},
	$$ 
	and, for $(s,x,t,y) \in \Theta$,  
	\begin{align} 
		\wr_{s,t}(x,y):=x- \Eb_{s,t}(y) \in \R^2,
		~~\mbox{with}~~
		\Eb_{s,t}(y):= \left(\begin{array}{cc}1&0\\ -(A_{t}-A_{s})&1\end{array}\right)y \in \R^2.
	\label{eq:def_w_xy}
	\end{align}

	\begin{Assumption}\label{hyp: stand ass}
		Let $\beta_0, \beta_1, \beta_2 \ge 0$ be the constants in Assumption \ref{hyp:A}.  Then, 		
		
		\noindent $\mathrm{(i)}$ We have 
		$$ 
			\beta'_1 := \beta_1-\beta_0 > -1,
			~~~~
			\beta'_2 := \beta_2 - \beta_0 > -1. 
		$$
Moreover, the coefficients $\mu$ and $\sigma$ are   continuous,
		and there exist constants $(\underline \ar, \bar \ar)\in \R^{2}$, $\br\in \R$, $C_{(\ref{eq : hyp holder coeff bar sigma})} > 0$
		and  $\alpha>0$ such that
		\begin{align} \label{eq : unif elliptic}
			|\mu| \le \br,
			~~
			0<\underline \ar\le \sigma^{2} \le \bar \ar, \;\mbox{ on }  [0,T] \x \R^2,
		\end{align}
		and 
		\begin{align}
			& \big|\sigma_{s}(x) -\sigma_{t}(y) \big|
			~\le~
			C_{(\ref{eq : hyp holder coeff bar sigma})} \Big(
			|t-s|^{\alpha}
			+
			\big| \wr_{s,t}(x,y)\big|^{\frac{2\alpha}{1+\beta'_1}}
			+ 
			\big|\wr_{s,t}(x,y) \big|^{\frac{2 \alpha}{1+\beta'_2}}
		 	\Big),
			 \label{eq : hyp holder coeff bar sigma} 
		\end{align}
 for all $(s,x,t,y) \in \Theta$.
	
		\noindent $\mathrm{(ii)}$ There exists a constant $C_{(\ref{eq : hyp holder coeff bar mu})}>0$ such that
		\begin{equation} \label{eq : hyp holder coeff bar mu} 	
			\big| \mu_{t}( x) - \mu_{t}( y) \big|
			\le
			C_{(\ref{eq : hyp holder coeff bar mu})} \Big(
				\big| x_{1}- y_{1}\big|^{\frac{2\alpha}{1+\beta'_1}}+\big|  x_{2}- y_{2}\big|^{\frac{2\alpha}{1+\beta'_2}}
			\Big),
			~\mbox{for all}~ (t,x,y) \in [0,T] \x \R^{2} \x \R^2.
		\end{equation}

	\end{Assumption}

	\begin{Example} 
		The condition \eqref{eq : hyp holder coeff bar sigma} holds for instance if $\sigma_{s}(x_1, x_2)$ depends only on $(s, x_{1})$ 
		and is H\"older with respect to $(s, x_1)$.
		It would also hold if it is of the form $\sigma_{s}(x)=\tilde \sigma_{s}(x_{1},x_{2}-A_{s}x_{1})$ for some H\"older continuous map $\tilde \sigma$.
		Indeed,  one has
		\begin{align*} 
			\big |x_{2}-A_{s}x_{1}-(y_{2}-A_{t}y_{1}) \big|
			&~=~
			\big| x_{2}-A_{s}x_{1}-(y_{2}-A_{t}y_{1})-A_{s}(y_{1}-x_{1})+A_{s}(y_{1}-x_{1}) \big| \\
			&~\le~
			\big|x_{2} -y_{2}+(A_{t}-A_{s})y_{1}) \big| + |A_{s}| \big|y_{1}-x_{1} \big|
			\\
			&~\le~
			\Big( 1+\max_{[0,T]}|A| \Big)  |\wr_{s,t}(x,y)|. 
		\end{align*}
	\end{Example}

	\begin{Remark}\label{rem: si holder}
		The case where the coefficient $\sigma$ is H\"older in the classical sense, i.e.
		$$ \big|\sigma_{s}(x) -\sigma_{t}(y) \big|
			~\le~
			C \Big(
			|t-s|^{\alpha}
			+
			\big| x_{1}-y_1\big|^{\frac{2\alpha}{1+\beta'_1}}
			+ 
			\big|x_{2}-y_{2} \big|^{\frac{2 \alpha}{1+\beta'_2}}
		 	\Big)
		$$
		can be tackled by combining the arguments below with those of 
		e.g.~\cite{sonin1967class}. 
		This will add additional exponentially growing terms in the estimates on {$\Phit$ in Proposition \ref{prop:existence_continuite_Phi} below}, which can be handled, to the price of adapted restrictions on the coefficients $(\beta_{i})_{0\le i\le 4}$. 
		We chose  the formulation of the conditions in \eqref{eq : hyp holder coeff bar sigma} for sake of simplicity. 
	\end{Remark}

\subsection{Heuristic derivation using a change of variables and the parametrix method}

	Let us consider the path-dependent SDE
	\begin{equation} \label{eq:SDE_XI}
		X_t 
		=
		X_0
		+
		\int_0^t \mu_{s} ( X_s , I_s ) d s 
		+ 
		\int_0^t \sigma_{t} ( X_s , I_s ) dW_s,
		~~~
		I_t = \int_{0}^{t} X_s d A_s,
		~~
		t \in [0,T],
	\end{equation}
	where $W$ is a Brownian motion.
	Assume that the above SDE has a solution $X$ such that $(X, I)$ is Markov.
	Then, to deduce a solution to the PPDE \eqref{eq:ppde}, it suffices to find the transition probability (density) function $f(s,x; t,y)$ of the Markov process $(X, I)$ from $(s,x)$ to $(t,y)$.
	When $t \longmapsto A_t$ is absolutely continuous, it is well-known that $(s,x) \longmapsto f(s,x; t,y)$ solves a Kolmogorov's backward PDE and that $(t,y) \longmapsto f(s,x; t,y)$ solves a Kolmogorov's forward PDE.
	One can then apply the classical parametrix method as in \cite[Section 4]{francesco2005class} to guess the expression of $f(s,x; t, y)$.
	
	\vspace{0.5em}
	
	In our setting where $A$ is not necessarily absolutely continuous, 
	it is no more possible to write the Kolmogorov's PDE for the transition probability (density) function of $(X, I)$.
	We therefore perform a change of variable and set 
	\begin{align}\label{eq: change variable tilde X} 
		\Xt_t := \AMp_t \left(\begin{array}{c} X_t \\ I_t \end{array}\right),
		~\mbox{with}~
		\AMp_t  := \left(\begin{array}{cc} 1 & 0\\ A_t &-1 \end{array}\right),
		~~ t \in [0,T].
	\end{align}
	Notice that $\AMp_t^{-1} = \AMp_t$ and that
	$\Xt$ is a diffusion process with dynamics 
	\begin{align}\label{eq:def_Xt} 
		\Xt_t
		&~=~
		\Xt_0 + \int_{0}^{t} \mut_{s}(\Xt_s) \Ap_s ds +\int_{0}^{t} \sigmat_{s}(\Xt_s) \Ap_s dW_s,
		~~ t \in [0,T],
	\end{align}
	where $\Ap$, $\mut: \R^2 \longrightarrow \R$ and $\sigmat: \R^2 \longrightarrow \R$ are defined by
	\begin{align}\label{eq: def mu sigma A} 
		\Ap_s := \left(\begin{array}{c} 1 \\  A_{s} \end{array}\right),
		~
		\mut_{s}(x)
		:=
		\mu_{s}(\AMp_s x)
		~\mbox{and}~
		\sigmat_{s}(x)
		:=
		\sigma_{s}(\AMp_s x),
		~(s,x)\in [0,T] \x \R^{2}.
	\end{align}
	The generator $\Lct$ of $\Xt$ is given by
	$$
		\Lct \phi(s,x) 
		~:=~
		\mut_{s}(x)~\Ap_s \cdot D\phi(s,x)+\frac12 \sigmat_{s}(x)^2 ~ \Tr\left[ \Ap_s (\Ap_s)^{\top} D^{2}\phi(s,x) \right],
	$$
	for smooth functions $\phi: [0,T] \x \R^2 \longrightarrow \R$.
	
	Assume that the SDE \eqref{eq:def_Xt} has a solution $\Xt$ which is Markovian and has a smooth transition probability density function $\ft(s,x; t,y)$, from $x$ at $s$ to $y$ at $t$, then 
	$(s,x) \longmapsto \ft(s,x; t,y)$ solves the  Kolmogorov backward equation
	\begin{align}\label{eq: pde ft} 
		\big(\partial_{s} + \Lct \big) \ft(s,x;t,y) = 0,
		~~\mbox{for}~ (s,x) \in [0,t)\x \R^{2}.
	\end{align}
	Notice that, in the above, the operator $\Lct$ {acts} on the first two arguments $(s,x)$ of $\tilde f(s,x; t, y)$.

	\vspace{0.5em}
	
	To construct the parametrix, we consider the following process, with volatility frozen at $(r,z) \in [0,T] \x \R^2$,
	$$
		\Xt^{r,z}_t ~:=~  \sigmat_r(z) \int_0^t  \Ap_s dW_s, 
		~~
		t \in [0,T].
	$$
	The corresponding generator $\Lct^{r,z}$ is then given by
	$$
		\Lct^{r,z} \phi(s,x) 
		~:=~
		\frac12 \sigmat_{r}(z)^2 ~ \Tr\left[ \Ap_s (\Ap_s)^{\top} D^{2}\phi(s,x) \right],
		~\mbox{for   smooth functions}~\phi.
	$$
	We further define $\ft_{r,z} (s,x; t, y)$ as the {corresponding} transition probability function from $(s,x)$ to $(t, y)$, for $(s,x,t,y) \in \Theta$.
	Notice that $\ft_{r,z}$ is explicitly given and that 
	$ y \longmapsto \ft_{r,z} (s,x; t,y)$ is the density function of the Gaussian random vector $x + \sigmat_r(z) \int_s^t \Ap_r d W_r$. 
	It satisfies
	\begin{align}\label{eq: pde fycirc} 
		\big (\partial_{s} + \Lct^{r, z} \big) \ft_{r,z}(s,x;t,y) = 0,\; \mbox{ for } (s,x) \in [0,t)\x \R^{2}.
	\end{align}  
 
 	Now, we employ the machinery of the parametrix method (see e.g. \cite[Chapter 1]{friedman2008partial} or \cite{francesco2005class}),
	taking $\ft_{t, y}(s,x; t,y)$ as parametrix, and expressing $\ft(s,x; t,y)$ in the following form:
	\begin{equation} \label{eq: def ft recursif}
		\ft(s,x;t,y)= \ft_{t, y}(s,x;t,y)+ \int_{s}^{t}\int_{\R^{2}} \ft_{r, z}(s,x;r,z) \Phit (r,z,t,y)dzdr,
	\end{equation}
	for some function $\Phit: \Theta \longrightarrow \R$.
	By \eqref{eq: pde ft} and \eqref{eq: pde fycirc}, one must have
	\begin{align*}
		0 
		~=~& 
		\big( \partial_{s} + \Lct \big) \ft (s,x;t,y) \\
		~=~&
		\big( \partial_{s} + \Lct \big) \ft_{t, y} (s,x;t,y)  + \big( \partial_{s} + \Lct \big)  \int_{s}^{t}\int_{\R^{2}} \ft_{r, z}(s,x;r,z)\Phit(r,z,t,y)dzdr \\
		~=~&
		\big(\Lct - \Lct^{t,y} \big) \ft_{t, y}(s,x; t,y) - \Phit(s,x; t, y)\\
		& + \int_s^t \int_{\R^2} \big( \Lct - \Lct^{t,z}\big) \ft_{r, z}(s,x; r,z) \Phit(r,z; t, y) dz dr.
	\end{align*}
	Therefore, $\Phit$ must satisfy
	\begin{equation} \label{eq:Phit_int_eq}
		\Phit(s,x; t, y)
		~=~
		\big(\Lct - \Lct^{t,y} \big) \ft_{t, y}(s,x; t,y)
		+
		\int_s^t \int_{\R^2} \big( \Lct - \Lct^{t,z}\big) \ft_{r, z}(s,x; r,z) \Phit(r,z; t, y) dz dr.
	\end{equation}
	In view of \eqref{eq:Phit_int_eq}, we obtain
	\begin{equation} \label{eq:def_Phit}
		\Phit(s,x; t, y) 
		~:=~
		\sum_{k=0}^{\infty} \Deltat_k(s,x; t,y),
	\end{equation}
	where $\Deltat _0(s,x; t,y) := \big(\Lct - \Lct^{t,y} \big) \ft_{t, y}(s,x; t,y) $,
	and
	\begin{align}\label{eq: def Delta L k} 
		\Deltat_{k+1}  (s,x;t,y) := \int_s^t   \int_{\R^2} \Deltat_0(s,x; r,z) \Deltat_{k}(r,z;t,y) dz dr, 
		~~k\ge 0.
	\end{align}
	
	Notice that $\Lct$, $\Lct^{t,y}$ and $\ft_{t,y}$ have explicit expressions.
	The main strategy of the classical parametrix method consists in checking that $\Phit$ in \eqref{eq:def_Phit} is well-defined and solves the integral equation \eqref{eq:Phit_int_eq}.
	Then, one defines $\ft$ by \eqref{eq: def ft recursif}, and check that it provides a solution to \eqref{eq: pde ft}.
	If $\ft$ is smooth, one can basically deduce that $\ft$ is the transition probability density function of $\Xt$ in \eqref{eq:def_Xt} by using the Feynman-Kac's formula.

	\vspace{0.5em}

	{The main difficulty here lies in the fact that $\ft$ is, in general, not smooth enough. For smoothness properties, we will therefore  turn back to the initial coordinates $(X,I)$ and define the candidate transition probability function $f$ of the process $(X, I)$ in \eqref{eq:SDE_XI}  through \eqref{eq:def_Phi}-\eqref{eq:def_f_transition_proba} below, and work on it directly. }

\subsection{Main results}
\label{sec:main_results}

	Under some conditions on the constants $\alpha$ and $(\beta)_{i = 0, \cdots, 4}$ given in Assumptions \ref{hyp:A} and \ref{hyp: stand ass},
	we will show that $\Phit : \Theta \longrightarrow \R$ is well-defined by \eqref{eq:def_Phit}-\eqref{eq: def Delta L k}.
	For $(r,z) \in [0,T] \x \R^2$, we can  then define $f_{r,z}$ and $\Phi$ by inverting the change of variables in \eqref{eq: change variable tilde X}: 
	\begin{equation} \label{eq:def_Phi}
		f_{r,z} (s,x;t,y)
		:=
		\ft_{r, \AMp_r z} (s,\AMp_{s}x;t,\AMp_{t}y),
		~~
		\Phi (s,x;t,y)
		:=
		\Phit (s,\AMp_{s}x;t,\AMp_{t}y).
	\end{equation}
	The corresponding candidate transition density  $f: \Theta \longrightarrow \R$ for $(X,I)$ is therefore:
	\begin{equation} \label{eq:def_f_transition_proba}
		f (s,x;t,y) ~:=~ f_{t, y}(s,x;t,y)+ \int_{s}^{t}\int_{\R^{2}} f_{r, z}(s,x;r,z) \Phi (r,z; t,y)dzdr,
	\end{equation}
	for all $(s,x,t,y) \in \Theta$.
	
	For any positive constant $a \in \R_+$ and $0 \le s < t \le T$, let us set 
	$$
		\Sigma_{s,t}(a) 
		~:=~
		a \left(   \begin{array}{cc} 
		t- s &  - \int_s^t (A_r - A_s) dr  \\
		- \int_s^t (A_r - A_s) dr & \int_s^t (A_r - A_s)^2 dr
		\end{array} \right).
	$$
	For $(r,z) \in [0,T] \x \R^2$, we write  $\Sigma_{s,t}(r,z) := \Sigma_{s,t}( \sigma^2_r(z) )$ for simplicity.
	Equivalently,
	\begin{equation} \label{eq:Sigma_st}
		\Sigma_{s,t} (r, z)
		~:=~
		\sigma^2_{r} \big(  z \big) \left(   \begin{array}{cc} 
		t- s &  - \int_s^t (A_r - A_s) dr  \\
		- \int_s^t (A_r - A_s) dr & \int_s^t (A_r - A_s)^2 dr
		\end{array} \right).
	\end{equation}
	Then, it is easy to check that $y \longmapsto f_{r,z}(s, x; t, y)$ is the density function of the Gaussian random vector
	$$
		\Big( x_1 + \sigma_r(z) (W_t - W_s), ~x_2 + \int_s^t \big( x_1 + \sigma_r(z) (W_u- W_s) \big) d A_u \Big)^{\top},
	$$
	so that, with $w:=\wr_{s,t}(x,y)$ {as in} \eqref{eq:def_w_xy},
	$$
		f_{r,z}(s,x;t,y)
		~=~ 
		\frac1{2\pi~ \det{ \Sigma_{s,t}(r, z)}^{\frac12}} \exp \Big( -\frac{1}{2} \big\langle \Sigma^{-1}_{s,t}(r, z) w, w \big\rangle \Big).
	$$
	Let us also define the Gaussian transition probability function ${f^{\circ}}: \Theta \longrightarrow \R$ by
	\begin{align} \label{eq: def fcirc} 
		f^{\circ} (s,x;t,y)
		~:=~
		\frac{1}{2\pi~ \det{\Sigmab_{s,t}( 4 \bar \ar) }^{\frac12}} \exp \Big( - \frac{1}{2} \big \langle \Sigma^{-1}_{s,t}( 4 \bar \ar) w, w \big \rangle \Big),
		~~(s,x,t,y) \in \Theta.
	\end{align}

	As a first main result, we show that $f$ is well-defined 	under some conditions on the coefficients $\alpha$ and $\beta_0$,
	and then provide some first regularity and bound estimates.

	\begin{Theorem}\label{thm:f_well_defined} 
		Let Assumption \ref{hyp:A}.$\mathrm{(i)}$ and Assumption \ref{hyp: stand ass}.$\mathrm{(i)}$ hold true.
		
		\vspace{0.5em}
		
		\noindent $\mathrm{(i)}$
		Assume that 
		\begin{equation} \label{eq:def_kappa0}
			\kappa_0 
			~:=~
			\frac{1-\beta_{0}}{2}\wedge( \alpha  - \beta_0 ) 
			~>~
			0.
		\end{equation}
		Then $\Phit$ in \eqref{eq:def_Phit}-\eqref{eq: def Delta L k} is well-defined, and so is $f: \Theta \longrightarrow \R$ in \eqref{eq:def_f_transition_proba}.
		Moreover, $f$ is continuous on $\Theta$, and there exists a constant $C>0$ such that
		\begin{equation} \label{eq:up_bound_f}
			{\big| f(s,x ;t,y) \big|}
			~\le~  C f^{\circ}(s,x;t,y),
			~~\mbox{for all}~
			(s,x,t,y )\in \Theta.
		\end{equation}

		\noindent $\mathrm{(ii)}$
		{Assume that} \eqref{eq:def_kappa0} holds and that 	
		\begin{align}\label{eq: kappa1} 
			\kappa_{1}
			~:=~
			\kappa_{0}+\frac{1-\beta_{0}}{2}
			~=~
			(1-\beta_{0})\wedge (\frac12 +\alpha-\frac32 \beta_{0})
			~>~
			0.
		\end{align}
		Then, {the partial derivative $(s,x; t, y)\in \Theta\mapsto \partial_{x_1} f(s,x; t, y)$  exists, is continuous on $\Theta$,}
		and, for some constant $C_{\eqref{eq: them estimation derive fb}}>0$,
		\begin{align}\label{eq: them estimation derive fb} 
			\big|\partial_{x_1} f(s,x; t,y) \big|
			~\le~
			\frac{C_{\eqref{eq: them estimation derive fb}}}{(t-s)^{1-\kappa_{1}}} f^{\circ}(s,x;t,y),
			~\mbox{for all}~
			(s,x,t,y) \in \Theta.
		\end{align}		
	\end{Theorem}

	Under further conditions, we can obtain more regularity of $f$ and then check that it is the transition probability function of the Markov process $(X,I)$.
	To be more precise, let us rephrase this in terms of path-dependent functionals.
	For $0 \le s < t \le T$, $\xr \in D([0,T])$ and $y \in \R^2$, we set
	\begin{align}\label{eq: def f bar on path} 
		(\fr,\fr^{\circ})(s,\xr;t,y)
		~:=~
		(f,f^{\circ}) \big(s, \xr(s), I_s(\xr); t,y \big ),
		~\mbox{with}~
		I_s(\xr) := \int_{0}^{s} \xr(r)dA_{r}.
	\end{align}
	We now fix $\ell : [0,T] \x \R^2 \longrightarrow \R$ and $g: \R^2 \longrightarrow \R$   such that, 
	for some constants $C_{{\ell,g}} > 0$ and $\alpha_{\ell} > 0$,
	\begin{equation} \label{eq:bound_g_l}
		\big| \ell(t, x) \big| + \big| g(x) \big| 
		~\le~
		C_{{\ell,g}} \exp \big( C_{{\ell,g}} |x| \big),
	\end{equation}
	and
	\begin{equation} \label{eq:holder_l}
		\big| \ell(t,x) - \ell(t, x') \big| 
		~\le~
		C_{{\ell,g}} \big( e^{C_{{\ell,g}}|x|} + e^{C_{{\ell,g}}|x'|} \big) \Big( |x_1 - x_1'|^{\frac{2 \alpha_{\ell}}{1+ \beta_1'}} + |x_2 - x_2'|^{\frac{2 \alpha_{\ell}}{1+ \beta_2'}} \Big),
	\end{equation}
	for all $t \in [0,T]$ and $x,x' \in \R^2$.
	In view of the upper-bound estimate of $f$ in \eqref{eq:up_bound_f}, we can then define
	\begin{equation} \label{eq:def_vr}
		\vr(s,\xr) 
		~:=
		\int_s^T\!\! \int_{\R^2} \!\! \ell(t,y) \fr(s, \xr; t, y) dy dt 
		~+
		\int_{\R^2} \!\! g(y) \fr(s, \xr; T, y) dy,
		~~
		(s,\xr) \in [0,T) \x D([0,T]).
	\end{equation}

	\begin{Remark} \label{rem:PPDE_f_rz}
		By its definition in \eqref{eq: def f bar on path}, it is straightforward to check that
		$$
			\partial_{\xr} \fr(s, \xr; t, y) 
			~=~ 
			\partial_{x_1} f \big( s, \xr(s), I_s(\xr); t, y \big).
		$$
		Similarly, let us define, for $(r,z), (t,y) \in [0,T] \x \R^d$,
		$$
			\fr_{r,z}(s, \xr; t,y) ~:=~ f_{r,z}(s, \xr(s), I_t(\xr); t,y),~(s, \xr) \in [0,t) \x D([0,T]).
		$$
		Then, the functional 
		$(s,\xr) \longmapsto \fr_{r,z}(s, \xr; t,y)$ is a classical solution to the PPDE
		\begin{equation} \label{eq:PPDE_f_rz}
			\partial_s \fr_{r,z}(s,\xr; t,y) 
			+
			\frac12 \sigma_r(z)^{2} \partial^2_{\xr \xr} \fr_{r,z} (s, \xr; t, y) = 0,
			~\mbox{for}~(s,\xr) \in [0,t) \x D([0,T]).
		\end{equation}
	\end{Remark}

	 \begin{Theorem} \label{thm:c12}
 		Let Assumptions \ref{hyp:A} and \ref{hyp: stand ass} hold true.
		Assume that \eqref{eq:def_kappa0}, \eqref{eq: kappa1}, \eqref{eq:bound_g_l} and \eqref{eq:holder_l} hold, and that there exists $\alpha_{\Phi}\in \R$ such that 
		$$
			0<\alpha_{\Phi}<\kappa_{0}\wedge \hat \alpha_\Phi\wedge \min\limits_{i=1,2} \frac{1+\beta'_{i}}{2},
			~~\mbox{with}~~
			\hat \alpha_\Phi
			:=
			\frac12- \beta_{0}-\frac{\widehat{\Delta \beta}}{2}  -\frac{ (\beta_{0}+1-2\alpha)^{+}}2,
		$$ 
		where
		\begin{align}\label{eq: def wide hat Delat beta} 
			\widehat{\Delta \beta}:=\max\left\{ \beta_{0}-\beta_{1}\;,\; \beta_{3}-\beta_{2}\right\},
		\end{align}
		and
		$$
			 {\min \Big(\frac{2\beta_{4}+1+\beta'_{1}}{1+\beta'_{2}}, 1 \Big) \min\{ \alpha_{\Phi},\alpha_{\ell},\alpha\}- \beta_0>0. }
		$$
		
		\noindent $\mathrm{(i)}$
		For each $(t,y)\in (0,T]\x \R^{2}$, the path-dependent functional $\fr(\cdot; t, y)$ belongs to $\Cb^{1,2}([0,t))$.
 	\vspace{0.5em}

		\noindent $\mathrm{ {(ii)}}$ $\vr\in  \Cb^{1,2}([0,T))$ and it solves the PPDE \eqref{eq:ppde}.
		Moreover, there exists $C>0$ such that, for all $(s,\xr) \in [0,T) \x D([0,T])$,

		\begin{equation} \label{eq:bound_dvr}
			\big| \partial_{\xr}   \vr(s,\xr) \big|
			~\le~
			\frac{C e^{C (|\xr_s| + |I_s(\xr)|)}}{(T-s)^{1-\kappa_1}},
			~\mbox{and}~
			\big| \partial_{t}   \vr(s,\xr) \big|
			+
			\big| \partial^{2}_{\xr}     \vr(s,\xr) \big|
			~\le~
			\frac{C e^{C (|\xr_s| + |I_s(\xr)|)}}{(T-s)^{1+\beta_0}}.
		\end{equation}

If, in addition, $g: \R^2 \longrightarrow \R$ is continuous, then $\vr$ is the unique classical solution to the PPDE \eqref{eq:ppde} satisfying 
		\begin{align}\label{eq: borne croissance vr et cond bord} 
			\lim_{t \nearrow T} \vr(t,\xr) = \bar g(\xr)
			{~~and ~   |\vr(s,\xr)| \le Ce^{C  (|\xr_s| + |I_s(\xr)|)}}
		\end{align}
		{for all $(s,\xr) \in [0,T]\x C([0,T])$, for some $C>0$.}
		\vspace{0.5em}
		
		\noindent $\mathrm{ {(iii)}}$ The SDE \eqref{eq:SDE_XI} has a unique weak solution $X$. 
		Moreover,   $(X, I)$ is a strong Markov process with transition probability given by $f$ and 
		\begin{equation} \label{eq:vr_as_exp}
			\vr(s, \xr) 
			~=~
			\E \Big[ \int_s^T \ell(X_t, I_t) dt + g(X_T, I_T) \Big| X_s = \xr(s), I_s = I_s(\xr) \Big],\;(s,\xr)\in [0,T]\x D([0,T]).
		\end{equation}

	\end{Theorem}

	\begin{Remark}
		To check the conditions on $\alpha$ and $\beta_i,~ i=0, \cdots, 4$ in Theorem \ref{thm:c12},
		let us stay in the setting of Example \ref{example : cas de A}.
		
		\vspace{0.5em}
		
		\noindent $\mathrm{(i)-(ii)}$ In these cases, $\kappa_{0}=\frac12 \wedge \alpha$, $\kappa_{1}=1 \wedge (\alpha+\frac12)$, $\hat \alpha_{\Phi}=\frac12-[\frac12-\alpha]^{+}=\alpha\wedge \frac12$ and we can choose $\alpha_{\Phi}\in (0,\frac12\wedge \alpha)$.

		\noindent $\mathrm{(iii)}$  In this case, $\kappa_{0}=\frac{1-2(\gamma_{1}-\gamma_{2})}2 \wedge (\alpha-2(\gamma_{1}-\gamma_{2}))$ which requires that $\gamma_{1}-\gamma_{2}< \frac12\wedge \frac{\alpha}{2}$ to ensure that $\kappa_{0}>0$. Then, $\kappa_{1}>0$ and  $\hat \alpha_{\Phi}=\frac12- 3(\gamma_{1}-\gamma_{2})-[\gamma_{1}-\gamma_{2}+\frac12-\alpha]^{+}$. If $2\alpha/(1+\beta'_{1})\le 1$, then $\alpha\le 1/2$ and therefore $\gamma_{1}-\gamma_{2}+\frac12-\alpha\ge 0$. In this case, we can choose  $\alpha_{\Phi}\in (0,\alpha- {4}(\gamma_{1}-\gamma_{2}))$ if $\gamma_{1}-\gamma_{2}<\alpha/ {4}$. If $2\alpha/(1+\beta'_{1})> 1$, then $(\mu,\sigma)$ does no depend on its first argument, and the different cases can also be treated explicitly, leading to a suitable $\alpha_{\Phi}$ when $\gamma_{1}-\gamma_{2}$ is small enough.
	\end{Remark}

	The conditions in Theorem \ref{thm:c12} ensure that $f$ is smooth enough, so that one can basically apply the Feynman-Kac's formula to justify that it is the transition probability function of a Markov process.
	It can then be  used to prove that the wellposedness (existence and uniqueness) of the SDE \eqref{eq:SDE_XI}.
	If one already knows that the SDE \eqref{eq:SDE_XI} has  a unique  a weak solution, 
	then one can  {rely on} Theorem \ref{thm:vr_uniqueX}  below,  {which requires}  less technical conditions on $A$ and {$(\mu,\sigma)$}, to check that $f$ is the corresponding transition probability function. In this case, the path-dependent functional $\vr$ defined above may only be $\Cb^{0,1}([0,T))$, but it is enough to  deduce that  \eqref{eq:vr_as_exp} holds, and obtain it's It\^{o}-Dupire's decomposition, whenever it satisfies for instance one of the conditions a.~or b.~of Theorem \ref{thm:vr_uniqueX} .

	\begin{Theorem}\label{thm:vr_uniqueX} 
		Let Assumption \ref{hyp:A}.(i) and Assumption \ref{hyp: stand ass}.(i) hold true,
		and assume that the SDE \eqref{eq:SDE_XI} has a unique weak solution, so that the corresponding process $( X, I)$ is a strong Markov process.
		
		\vspace{0.5em}
		
		\noindent $\mathrm{(i)}$ Assume in addition that \eqref{eq:def_kappa0} holds true so that $f$ is well defined.
		Then, $f$ is the transition probability function of $( X, I)$,
		and \eqref{eq:vr_as_exp} {holds whenever \eqref{eq:bound_g_l} does}.
		
		\vspace{0.5em}
		
		\noindent $\mathrm{(ii)}$ Assume that \eqref{eq:def_kappa0}, \eqref{eq: kappa1} {and \eqref{eq:bound_g_l}} hold. {Then,}
		  $\vr\in \Cb^{0,1}([0,T))$. 
		Suppose in addition that one of the following  holds:
		\begin{itemize}
			\item[{\rm (a)}] there exists $C_{\eqref{eq: cond vr Lip}}>0$ such that 
			\begin{align}\label{eq: cond vr Lip} 
				\big|\vr(t,\xr)-\vr(t,\xr') \big| 
				~\le~
				C_{\eqref{eq: cond vr Lip}}  \int_{0}^{t} |\xr(r)-\xr'(r)| d|A|_{r},
			\end{align}
			for all $t \in [0,T]$, $\xr,\xr'\in D([0,T])$ such that $\xr(t) = \xr'(t)$, in which $|A|$ denotes the total variation of $A$.
			\item[{\rm (b)}] $A$ is monotone and $  0 < \frac{1+\beta_{2}-\beta_{0}}{2+4\beta_{4}} < 1-\frac{\beta_3-\beta_{2}+\beta_{0}}{2}$.
		\end{itemize}
		Then, 
		\begin{align}\label{eq: Ito vr} 
			\vr(t, X)=\vr(0,X)+\int_{0}^{t} \partial_{\xr} \vr(s, X) \sigmab_s (X) dW_s
			-\int_0^t \bar \ell_{{s}}( X) ds,
			~~t \in [0,T]. 
		\end{align}

	\end{Theorem}

	\begin{Remark}
		When  $b/\sigma$ is     bounded,  and $\sigma$ is Lipschitz in its space variable  in the sense that, for some constant $C_{\eqref{eq: coeff Lipschitz}}>0$, 
		\begin{equation} \label{eq: coeff Lipschitz}
			\big|\sigma_{s}(x)-\sigma_{s}(x') \big|
			~\le~
			C_{\eqref{eq: coeff Lipschitz}}|x-x'|,
			~~s \in [0, T],\; x,x'\in \R^{2},
		\end{equation}
		with $(\sigma_{s}(0))_{s\le T}$ bounded, 
		then the SDE \eqref{eq:SDE_XI} has a   unique weak solution.  
	\end{Remark}

	\begin{Remark} 
		To check the conditions in Theorem \ref{thm:vr_uniqueX}.$\mathrm{(ii).(b)}$, let us consider the situations of Example \ref{example : cas de A}.
		
		\vspace{0.5em}

		\noindent $\mathrm{(i)-(ii)}$ 
		In these cases, $\frac{1+\beta_{2}-\beta_{0}}{2+4\beta_{4}} = \frac12$ and $1-\frac{\beta_3-\beta_{2}+\beta_{0}}{2} = 1$,
		so that the conditions inTheorem \ref{thm:vr_uniqueX}.$\mathrm{(ii).(b)}$ hold true.

		\vspace{0.5em}

		\noindent $\mathrm{(iii)}$ 
		In this case, $\frac{1+\beta_{2}-\beta_{0}}{2+4\beta_{4}} = \frac12-\frac{\gamma_{1}-\gamma_{2}}{1+2\gamma_{2}}$ and
		$1-\frac{\beta_3-\beta_{2}+\beta_{0}}{2} =1-2(\gamma_{1}-\gamma_{2})$.
		Therefore, the conditions in Theorem \ref{thm:vr_uniqueX}.$\mathrm{(ii).(b)}$ hold true when $\gamma_{1}-\gamma_{2}$ is small enough.\\
	\end{Remark}

\section{Proofs}\label{sec: proofs}
 
This section is devoted to the proof of Theorems \ref{thm:f_well_defined}, \ref{thm:c12} and \ref{thm:vr_uniqueX}.

\subsection{A priori estimates}

	Recall that, with $w :=\wr_{s,t}(x,y)$ (see \eqref{eq:def_w_xy}),
	$$
		f_{r,z}(s,x;t,y)
		~:=~ 
		\frac1{2\pi \det{ \Sigma_{s,t}(r, z)}^{\frac12}} \exp \Big( -\frac{1}{2} \big\langle \Sigma^{-1}_{s,t}(r, z) w, w \big\rangle \Big),
	$$
	where
 	$$
		\Sigma_{s,t} (r, z)
		~:=~
		\sigma^2_{r} \big(  z \big) \left(   \begin{array}{cc} 
		t- s &  - \int_s^t (A_r - A_s) dr  \\
		- \int_s^t (A_r - A_s) dr & \int_s^t (A_r - A_s)^2 dr
		\end{array} \right).
	$$
	By direct computation, one has
	\begin{equation} \label{eq: det sigma bar}
		\det{\Sigmab_{s,t}(r, z) } 
		~=~
		\sigma^4_{r} \big( z \big) (t-s)^2 m_{s,t},
	\end{equation}
	and hence
	\begin{align}\label{eq: def Sigmab -1} 
		\Sigmab_{s,t}^{-1} (r, z)
		~=~ 
		\sigma^{-2}_{r}\big(z \big)
		\left(   \begin{array}{cc} 
		\frac{1}{t-s} \frac{\tilde m_{s,t}}{m_{s,t}} &  \frac{\int_s^t (A_r - A_s) dr}{(t-s)^2 m_{s,t}}  \\
		\frac{\int_s^t (A_r - A_s) dr}{(t-s)^2 m_{s,t}} & \frac{1}{(t-s) m_{s,t}}
		\end{array} \right).
	\end{align}

	The following quantities will play an important role in our analysis.  
	For $i=1,2$, $(r,z) \in \R^{2}$ and $(s,x,t,y)\in \Theta$, with $w :=\wr_{s,t}(x,y)$ (see \eqref{eq:def_w_xy}),
	we compute that
	\begin{equation} \label{eq:Dx1fb}
		\partial_{x_i} f_{r,z} (s,x; t,y) =  f_{r,z} (s,x; t, y) \Big( - \big(\Sigmab^{-1}_{s,t}(r, z) w \big)_i \Big),
	\end{equation}
	\begin{equation} \label{eq:Dx1x2fb}
		\partial^2_{x_1 x_i} f_{r,z} (s,x; t,y) 
		=
		f_{r,z}(s,x; t,y) 
		\Big( \big( \Sigmab^{-1}_{s,t}(r,z)  w \big)_1 \big( \Sigmab^{-1}_{s,t}(r,z)  w \big)_i 
			- \big( \Sigmab^{-1}_{s,t} (r,z) \big)_{1,i} \Big),
	\end{equation}
	and
	\begin{align} \label{eq:Dx1x1x2fb}
		\frac{\partial^3_{x_1 x_1 x_i} f_{r,z} (s,x; t,y)}{f_{r,z} (s,x; t, y)}
		=~
		& 2 \big( \Sigmab^{-1}_{s,t}(r,z) w \big)_1 \big( \Sigmab^{-1}_{s,t}(r,z) \big)_{1,i}  + \big( \Sigmab^{-1}_{s,t}(r,z) w \big)_i \big( \Sigmab^{-1}_{s,t}(r,z) \big)_{1,1}  \nonumber \\
		&
		- \big( \Sigmab^{-1}_{s,t}(r,z) w \big)_1^2 \big( \Sigmab^{-1}_{s,t}(r,z) w \big)_i.
	\end{align}
	
	Let us first provide some estimations in the following lemma. 

	\begin{Lemma} \label{lemma:Sigamw} 
		Let { Assumption \ref{hyp:A}.(i) hold.}
		Then, there exists constants 
		$C_{\eqref{eq: norne Sigma-1 w 1}}$, $C_{\eqref{eq: norne Sigma-1 w 2}}$, 
		$C_{{\eqref{eq: borne Sigma -1 11} }}$,
		$C_{{\eqref{eq: borne Sigma -1 12} }} > 0$,
		such that, for all $(s,x,t,y)\in \Theta$ and  $(r,z) \in [0,T] \x \R^{2}$, with $w := \wr_{s,t}(x,y)$ (see \eqref{eq:def_w_xy}), we have 
	\begin{align}\label{eq: norne Sigma-1 w 1} 
		\Big| \big(\Sigmab^{-1}_{s,t} (r, z) w \big)_1 \Big|
		&~=~
		\Big| \frac{\partial_{x_1} f_{r,z} (s,x; t,y)}{f_{r,z} (s,x; t, y)} \Big|
		~\le~
		\frac{C_{\eqref{eq: norne Sigma-1 w 1}}}{(t-s)^{\frac{1 + \beta_0}{2}}} \sqrt{\big \langle \Sigmab^{-1}_{s,t}(r,z) w, w \big \rangle},\\
		\label{eq: norne Sigma-1 w 2} 
		\Big| \big(\Sigmab^{-1}_{s,t} (r, z) w \big)_2 \Big|
		&~\le~
		\frac{C_{\eqref{eq: norne Sigma-1 w 2}} }{(t-s)^{\frac{1 + \beta_3}{2}}} \sqrt{\big \langle \Sigmab^{-1}_{s,t}(r,z) w, w \big \rangle}, \\
		\Big|  \big( \Sigmab^{-1}_{s,t} (r,z) \big)_{1,1} \Big| 
		&~\le~  \frac{C_{{\eqref{eq: borne Sigma -1 11} }}}{(t-s)^{1+ \beta_{0} }}\label{eq: borne Sigma -1 11}, \\
		\Big|  \big( \Sigmab^{-1}_{s,t} (r,z) \big)_{1,2} \Big| 
		&~\le~  \frac{C_{{\eqref{eq: borne Sigma -1 12} }}}{(t-s)^{1+\frac{\beta_{0}+\beta_{3}}{2}}}\label{eq: borne Sigma -1 12}.
	\end{align}
	\end{Lemma}
	
	\begin{proof} The bound in \eqref{eq: borne Sigma -1 11} and \eqref{eq: borne Sigma -1 12} are  immediate consequences of Assumption \ref{hyp:A}.(i) and  \eqref{eq: def Sigmab -1}, up to appealing to  Cauchy-Schwarz's inequality for the latter. 
		By direct computation, one has 
		\begin{align*}
			\sigma_{t}(z)^{2}\big \langle \Sigmab^{-1}_{s,t}(r, z) w, w \big \rangle
			&= 
			\frac{1}{(t-s)^2 m_{s,t}} \int_s^t \Big( (A_r - A_s)^2 w_1^2 + 2 (A_r - A_s) w_1 w_2 + w_2^2 \Big) dr \\
			&=
			\frac{1}{(t-s)^2 m_{s,t}} \int_s^t \big( (A_r - A_s)w_1 + w_2 \big)^2 dr.
		\end{align*}
		Hence, using Assumption \ref{hyp: stand ass} and \eqref{eq : hyp vitesse explo},
		\begin{align*}
			\underline\ar\Big| \big(\Sigmab^{-1}_{s,t}(r, z) w \big)_1 \Big|
			&\le \sigma_{t}( z)^{2}\Big| \big(\Sigmab^{-1}_{s,t}(r, z) w \big)_1 \Big| 
			\\
			&= 
			\frac{1}{(t-s)^2 m_{s,t}} \Big| \int_s^t \big(A_r -A_s\big) \big( (A_r - A_s) w_1 + w_2 \big) dr \Big| \\
			&\le
			\sqrt{ \frac{  \tilde m_{s,t}}{ (t-s)  m_{s,t}} }\sqrt{\bar \ar\big \langle \Sigmab^{-1}_{s,t}(r, z) w, w \big \rangle}
			\le
			\frac{\sqrt{\bar \ar C_\eqref{eq : hyp vitesse explo}}}{(t-s)^{\frac{1 + \beta_0}{2}}} \sqrt{\big \langle \Sigmab^{-1}_{s,t}(r, z) w, w \big \rangle}.
		\end{align*}
		
		Similarly, using the above, Assumption \ref{hyp: stand ass} and Cauchy-Schwarz inequality, implies that 
		\begin{align*}
			\underline\ar \Big| \big(\Sigmab^{-1}_{s,t}(r, z) w \big)_2 \Big| 
			&\le \sigma_{t}( z)^{2}\Big| \big(\Sigmab^{-1}_{s,t}(r,z) w \big)_2 \Big| 
			\\&= 
			\Big|  \frac{1}{(t-s)^2 m_{s,t}} \int_s^t  \big( (A_r - A_s) w_1 + w_2 \big) dr \Big| \\
			&\le
			\frac{\sqrt{\bar \ar C_{\eqref{eq:order_m}}}}{(t-s)^{\frac{1 + \beta_3}{2}}} \sqrt{\big \langle \Sigmab^{-1}_{s,t}(r,z) w, w \big \rangle}.
		\end{align*}
	\end{proof}

	As usual, an important step consists in providing a suitable upper-bound on the parametrix density.
	Recall that $y \longmapsto f^{\circ} (s,x;t, y)$ defined in \eqref{eq: def fcirc} is a Gaussian density function on $\R^{2}$.

	\begin{Lemma} \label{lem: control density} 
	Let Assumption \ref{hyp:A}.(i) hold.
	Then, there exists $C_{\eqref{eq: def varomega}} >0$ such that, for all $(s,x,t,y)\in \Theta$ and  $(r,z) \in [0,T] \x \R^{2}$, we have  
		\begin{align}\label{eq: majo fy} 
			f_{r,z} (s,x;t,y)
			~\le~
			\varpi(s,x;t,y) ~f^{\circ} (s,x;t,y),
		\end{align}
		in which  $\varpi:=\varpi^{1}\varpi^{2}$ with 
		\begin{align}\label{eq: def varomega} 
			\left\{\begin{array}{l}
			\varpi^{1}(s,x;t,y)
			~:=~
			C_{\eqref{eq: def varomega}} 
			\exp\left(
				- \frac{1}{C_{\eqref{eq: def varomega}}}
				\Big(
					\frac{|w_1|^2}{(t-s)^{1+\beta'_1}}
					+
					\frac{|w_2|^2}{(t-s)^{1+\beta'_2}}
				\Big)
			\right), \\
			\varpi^{2}(s,x;t,y)
			~:=~
			\exp\Big(  -\frac12 \langle \Sigmab^{-1}_{s,t}(4 \bar \ar) w, w \big \rangle
			\Big),
			\end{array}
			\right.
		\end{align}
		where $w: =\wr_{s,t}(x,y)$ as defined in \eqref{eq:def_w_xy}.	
\end{Lemma}
	\begin{proof} Let us first observe that $m_{s,t}=\tilde m_{s,t}-[(t-s)^{-1}\int_{s}^{t}(A_{r}-A_{s})ds]^{2}$, so that 
	 the right-hand side of \eqref{eq : hyp vitesse explo} is equivalent to 
	\begin{align*} 
		\left(\frac1{t-s}\int_{s}^{t}(A_{r}-A_{s})ds\right)^{2}
		~\le~
		\tilde m_{s,t} \left(1-\frac{(t-s)^{\beta_{0}}}{C_{\eqref{eq : hyp vitesse explo}}}\right).
	\end{align*}
	Note that, upon changing the value of $C_{\eqref{eq : hyp vitesse explo}}$, one can assume that 
$C_{\eqref{eq : hyp vitesse explo}}\ge 2T^{\beta_{0}}$. 
	Hence,  using the inequality $2ab\le a^{2}+b^{2}$ for $a,b\in \R$, 
	\begin{align*}
		2\left|\frac{\int_{s}^{t}(A_{r}-A_{s})w_{1}w_{2}ds}{(t-s)^{2}m_{s,t}}\right|
		&\le 2 \left[\tilde m_{s,t} \left(1-\frac{(t-s)^{\beta_{0}}}{C_{\eqref{eq : hyp vitesse explo}}}\right)\right]^{\frac12}
		\frac{w_{1}w_{2}}{(t-s)m_{s,t}}\\
		&\le \left(1-\frac{(t-s)^{\beta_{0}}}{C_{\eqref{eq : hyp vitesse explo}}}\right)^{\frac12}\left\{ \frac{\tilde m_{s,t}}{(t-s)m_{s,t}} |w_{1}|^{2}
		+\frac{1}{(t-s)m_{s,t}}|w_{2}|^{2}\right\}.
	\end{align*}
	Combining the above with \eqref{eq : hyp vitesse explo}-\eqref{eq:order_m} and Assumption \ref{hyp: stand ass}  implies that  
	\begin{align}
		\big \langle \Sigmab^{-1}_{s,t}(r, z) w, w \big \rangle 
		&\ge
		\frac{1}{\sigma_{r}(z)^{2}} \left(\frac{\tilde m_{s,t}}{(t-s)m_{s,t}}|w_{1}|^{2}- 2\left|\frac{\int_{s}^{t}(A_{r}-A_{s})w_{1}w_{2}ds}{(t-s)^{2}m_{s,t}}\right|+\frac{ 1}{(t-s)m_{s,t}}|w_{2}|^{2} \right)
		\nonumber \\
		&\ge
		\frac{C}{\bar \ar} \left(  \frac{|w_1|^2}{(t-s)^{1+\beta_1-\beta_{0}}}
		+ \frac{|w_2|^2}{(t-s)^{1+\beta_2-\beta_{0}}} \right),\label{eq: mino forme quadra par diago}
	\end{align}
	for some $C>0$ that does not depend on $(s,x,t,y)$.
	The required result then follows from obvious algebra and Assumption \ref{hyp: stand ass}.
	\end{proof}

	\begin{Lemma} \label{lem: fcirc 1/2}
		Let Assumption \ref{hyp:A}.(i) hold.
		Let us define the transition density function $f^{\circ,\frac12}$ by,   for $(s,x,t,y)\in \Theta$,
		\begin{align}\label{eq: def fcirc1/2}
			f^{\circ,\frac12}(s,x;t,y)
			:=
			\frac{1}{2 \pi~ \det{\Sigmab_{s,t}(8 \bar \ar) }^{\frac12}} 
			\exp \Big( - \frac12 \big \langle \Sigmab^{-1}_{s,t}(8 \bar \ar)\wr_{s,t}(x,y) , \wr_{s,t}(x,y)  \big \rangle\Big).
		 \end{align}
		Then, there exists $C_{[\ref{lem: fcirc 1/2}]} >0$ such that, for all $(s,x,t,y)\in \Theta$ and  $x'\in \R^{2}$ satisfying 
		\begin{align}\label{eq: hypo dist x x'} 
			|x_1 - x_1'|^{\frac{1}{1+\beta'_1}} + |x_2 - x_2'|^{\frac{1}{1+\beta'_2}} 
			~\le~
			(t-s)^{1/2},
		\end{align}
		we have  
		\begin{align}\label{eq: majo fcirc par fcirc 1/2} 
			f^{\circ} (s,x';t,y)
			\le
			C_{[\ref{lem: fcirc 1/2}]} f^{\circ,\frac12} (s,x;t,y)
		\end{align}
		and 
		\begin{align}\label{eq: borne norme compensee par varpi} 
			\left(\frac{|\wr_{s,t}(x',y)_{1}|^{2}}{(t-s)^{{1+\beta'_{1}}}}
			+\frac{|\wr_{s,t}(x',y)_{2}|^{2}}{(t-s)^{{1+\beta'_{2}}}}\right)
			\varpi^{1}(s,x,t,y)
			\le
			C_{[\ref{lem: fcirc 1/2}]} .
		\end{align}
	\end{Lemma}

	\begin{proof} Set $w:=\wr_{s,t}(x,y)$ and $w':=\wr_{s,t}(x',y)$. First observe that 
	$$
		\left(\langle \Sigmab^{-1}_{s,t}(\bar \ar) w,w\rangle\right)^{\frac12}
		\le
		\left(\langle \Sigmab^{-1}_{s,t}(\bar \ar)(w-w'),(w-w')\rangle\right)^{\frac12}
		+
		\left(\langle \Sigmab^{-1}_{s,t}(\bar \ar)w',w'\rangle\right)^{\frac12}. 
	$$
	Using that $2ab\le 2 a^{2}+2b^{2}$ for $a,b\ge 0$, we deduce that 
	\begin{align*}
		-\langle \Sigmab^{-1}_{s,t}(\bar \ar)w',w'\rangle
		&\le -\frac12 \langle \Sigmab^{-1}_{s,t}(\bar \ar)w,w\rangle + \langle \Sigmab^{-1}_{s,t}(\bar \ar)(w-w'),(w-w')\rangle.
	 \end{align*}
	Now, by the same arguments as in the proof of Lemma \ref{lem: control density} and Assumption \ref{hyp: stand ass}, we have 
	\begin{align*}
		\big\langle \Sigmab^{-1}_{s,t}(\bar \ar)(w-w'),(w-w') \big \rangle
		&\le 
		2 \bar \ar \left(  \frac{|x_1-x_{1}'|^2}{(t-s)^{1+\beta'_1}}
				+ \frac{|x_2-x_{2}'|^2}{(t-s)^{1+\beta'_2}} \right)\le 4 \bar \ar,
	\end{align*}
	in which we used \eqref{eq:def_w_xy} and our assumption \eqref{eq: hypo dist x x'}. This proves \eqref{eq: majo fcirc par fcirc 1/2}. The assertion \eqref{eq: borne norme compensee par varpi}  is proved similarly, upon interchanging the role of $x$ and $x'$.
	\end{proof}

\subsection{Wellposedness of $\Phit$}

	In this section, we prove that $\Phit$  in  \eqref{eq:def_Phit}-\eqref{eq: def Delta L k}   is well defined. 
	Recall that $f^{\circ}$ is defined in \eqref{eq: def fcirc}, and let us define 
	\begin{align*}
		\ft^{\circ}(s,x;t,y)
		~:=~
		f^{\circ}(s,\AMp_{s}x;t,\AMp_{t}y),
		~~
		(s,x,t,y)\in \Theta.
	\end{align*}
	Noticing that $\AMp=\AMp^{-1}$, and recalling that
	$$
		f_{r, z}(s,x;t,y)~:=~\ft_{r, \AMp_{r}z}(s,\AMp_{s}x;t,\AMp_{t}y),
	$$
	it is straightforward to check that
	\begin{align}
		\partial_{x_{1}} f_{r,z}(s,x;t,y)
		&~=~
		\Ap_{s} \cdot D_x \ft_{r, \AMp_{r} z}(s,\AMp_{s}x;t, \AMp_{t}y), 
		\label{eq: lien derive f bar f} \\
		\partial^{2}_{x_{1}x_{1}} f_{r,z}(s,x;t,y)
		&~=~
		{\rm Tr} \Big[ \Ap_{s} \big(\Ap_s \big)^{\top}  D^{2}_{xx} \ft_{r, \AMp_{r} z}(s,\AMp_{s}x; \AMp_{t}y) \Big].
		\label{eq: lien derive seconde f bar f}
	\end{align}

	\begin{Lemma}\label{lem : estime L-Lfy} \
		{Let the conditions of Theorem \ref{thm:f_well_defined}.(i) hold.}
		Then, there exist a constant $C_{\eqref{eq: estimee L-Lfy}}>0$ such that 
		\begin{align}\label{eq: estimee L-Lfy} 
			\big| \big(\Lct - \Lct^{t, \yt} \big) \ft_{t, \yt}(s,\xt;t,\yt) \big|
			&~\le~ 
			\frac{C_{\eqref{eq: estimee L-Lfy}}}{(t-s)^{1-\kappa_0}} \ft^{\circ}(s,\xt;t,\yt),
			~\mbox{for all}~
			(s,\xt,t,\yt)\in \Theta,
		\end{align}
		{in which $\kappa_{0}$ is defined in \eqref{eq:def_kappa0}.}
	\end{Lemma} 
	\begin{proof} For simplicity, we assume that $t-s\le 1$, the case $t-s>1$ being trivially handled.  
	Let us denote
	\begin{equation} \label{eq:def_xbyb}
		x :=   \AMp_{s} \xt
		~~\mbox{and}~
		y :=  \AMp_{t} \yt.
	\end{equation}
	
	$\mathrm{(i)}$ Using \eqref{eq: def mu sigma A}  and \eqref{eq: lien derive f bar f}, we first estimate 
	\begin{align*} 
		I_{1}
		~:=~
		  \mut_{s}(\xt) ~\Ap \cdot D_x \ft_{t,\yt}(s,\xt;t,\yt)
		~=~
		 \mu_{s}(x) ~ \partial_{x_1} f_{t, y} (s, x; t, y).
	\end{align*}
	Then, by  Assumption \ref{hyp: stand ass}, Lemmas \ref{lemma:Sigamw} and \ref{lem: control density}, 
	it follows that 
	\begin{align}
		| I_{1}|
		~\le~
		 \frac{\br C_{\eqref{eq: norne Sigma-1 w 1}}}{(t-s)^{\frac{1 + \beta_0}{2} }}
		\sqrt{\big \langle \Sigmab^{-1}_{s,t}(t,y) {w}, {w} \big \rangle} 
		f_{t,y}(s, x;t, y)
		~\le~
		\frac{\br C_{\eqref{eq: norne Sigma-1 w 1}} C_{\eqref{eq: def const 1}}}{ (t-s)^{1 - \frac{1-\beta_0 }2}}
		f^{\circ}(s, x; t, y),
		\label{eq: borne I1 estimation L-L} 
	\end{align}
	in which $ w := \wr_{s,t}({x},{y})$ and, with $w' :=\wr_{{s',t'}}({x'},{y'})$,
	\begin{align}\label{eq: def const 1} 
		C_{\eqref{eq: def const 1}}
		~:=
		\sup_{(s',x',t',y',z')\in \Theta\x \R^{2}} \sqrt{\big \langle \Sigmab^{-1}_{s',t'}(t',z') w', w' \big \rangle} \varpi(s',x'; t',y')
		~<~
		\infty.
	\end{align}
 	
	\noindent $\mathrm{(ii)}$ Using \eqref{eq: def mu sigma A}  and \eqref{eq: lien derive seconde f bar f}, we now estimate 
	\begin{align*}
		I_{2}
		&~:=~
		\Tr \Big[
			\big( \sigmat^2_{t}(\yt) - \sigmat^2_{s}(\xt)  \big) 
			\Ap_s (\Ap_s)^{\top} D^{2}_{xx}  \ft_{t,\yt}(s,\xt; t,\yt)
		\Big]
		=
		\big[ \sigma^2_{t}(y) - \sigma^2_{s}(x) \big] \partial^2_{x_1 x_1} f_{y} (s, x; t, y).
	\end{align*} 
	By \eqref{eq:Dx1x2fb} and Lemma \ref{lemma:Sigamw}, one obtains
	\begin{align*}
		|I_2| 
		&~\le~
		 \big[ \sigma^2_{t}(y) - \sigma^2_{s}(x) \big] 
		\frac{(C_{\eqref{eq: norne Sigma-1 w 1}})^{2}\vee C_{(\ref{eq: borne Sigma -1 11})}}{(t-s)^{1+\beta_0}} \Big( \big \langle \overline \Sigma_{s,t}^{-1} (t,y) w, w \big \rangle + 1   \Big)  f_{y} (s, x; t, y).
	\end{align*}
	Recalling \eqref{eq:def_w_xy}, \eqref{eq : hyp holder coeff bar sigma}, \eqref{eq : unif elliptic} and Lemma \ref{lem: control density}, it follows that, for some $C>0$ that does not depend on $(s,x,t,y)$ and $z\in \R^{2}$, 
	\begin{align*}
		|I_2|
		&~\le~
		C 
			\frac{1}{(t-s)^{1+\beta_0-\alpha}}\Big(
			1
			+
			\big| \wr_{s,t}(x,y)\big|^{\frac{2\alpha}{1+\beta'_1}}
			+ 
			 \big|\wr_{s,t}(x,y) \big|^{\frac{2\alpha}{1+\beta'_2}}
			\Big)(\varpi^{1}  f^{\circ})(s, x; t, y),
	\end{align*}
	and we conclude by using the definition of $\varpi^{1} $ in \eqref{eq: def varomega}.
	\end{proof}

 	\begin{Proposition}\label{prop:existence_continuite_Phi} 
		{Let the conditions of Theorem \ref{thm:f_well_defined}.(i) hold.} Then, the  sum in \eqref{eq:def_Phit} is well-defined and there exists a constant $C_{\eqref{eq: estime vraie densite}}>0$ such that 
		\begin{align}\label{eq: estime vraie densite} 
			\big| \Phit(s,x;t,y) \big|
			&~\le~
			\frac{C_{\eqref{eq: estime vraie densite}}}{(t-s)^{1-\kappa_{0}}} \ft^{\circ}(s,x;t,y), 
			~\mbox{for all}~
			(s,x,t,y)\in \Theta.
		\end{align}
		Moreover, $\Phit$ is continuous on $\Theta$ and satisfies
		\begin{align}\label{eq: sol eq Phi} 
			\Phit(s,x;t,y)
			&=\Deltat_0(s,x; t,y)+\int_s^t   \int_{\R^2} \Deltat _0(s,x; r,z) \Phit(r,z;t,y) dz dr,
			~\mbox{for all}~
			(s,x,t,y)\in \Theta.
		\end{align}
	\end{Proposition}
 	\begin{proof} Let us recall that, if well-defined,
		\begin{equation*}  
		\Phit(s,x; t, y) 
		~:=~
		\sum_{k=0}^{\infty} \Deltat_k(s,x; t,y),
	\end{equation*}
	where $\Deltat _0(s,x; t,y) := \big( \Lct - \Lct^{t, y} \big) \ft_{t,y}(s,x; t,y) $,
	and
	\begin{align*} 
		\Deltat_{k+1}  (s,x;t,y) := \int_s^t   \int_{\R^2} \Deltat _0(s,x; r,z) \Deltat_{k}(r,z;t,y) dz dr, 
		~~k\ge 0.
	\end{align*}
	We already know from  Lemma \ref{lem : estime L-Lfy} that 
		\begin{align*}
			\big| \big( \Lct -\Lct^{t,y} \big) \ft_{t, y}(s,x;t,y) \big|
			&~\le~ 
			\frac{C_{\eqref{eq: estimee L-Lfy}}}{(t-s)^{1-\kappa_0}} 
			\ft^{\circ}(s,x;t,y),
		\end{align*}
	for all $(s,x,t,y)\in \Theta$. 
	By the same induction argument as in  \cite[proof of Proposition 4.1]{francesco2005class}, together with \eqref{eq: det sigma bar} and \eqref{eq:order_m}, we then deduce that 
	\begin{align}\label{eq: estime Delta k} 
		\big| \Deltat_{k}  (s,x;t,y) \big|
		~\le~
		\frac{M_{k}}{(t-s)^{1-k\kappa_{0}}} \ft^{\circ}(s,x;t,y)
		~\le~
		CM_{k} (t-s)^{k\kappa_{0}-2-\frac{\beta_{3}}{2}},
	\end{align}
	in which, $C>0$ does not depend on $(s,x,t,y)$ and $k$, and 
	$$
		M_{k}
		~:=~
		\frac{\{C_{\eqref{eq: estimee L-Lfy}}\Gamma(\kappa_{0})\}^{k}}{\Gamma(k\kappa_{0})},
	$$
	where $\Gamma$ denotes the Gamma function.  By dominated convergence, each map $\Deltat_{k}$ is continuous. 
	Then, the well-posedness of $\Phit$  follows from the fact that the power series $\sum_{k\ge 0} M_{k} u^{k}$ has a radius of convergence equal to $\infty$. 
	Continuity of $\Phit$ is a consequence of the absolute continuity of the series.

	\vspace{0.5em}

	It remains to prove \eqref{eq: sol eq Phi}. Note that, by the above,  
	$$
		\Phit(s,x;t,y)
		=
		\Deltat _0(s,x; t,y)+\sum_{k\ge 0 }\int_s^t   \int_{\R^2} \Deltat _0(s,x; r,z) \Deltat _k(r,z;t,y) dz dr
	$$
	and the family $\{(r,z)\in (s,t)\x \R^{2} \mapsto  \sum_{k=0}^{n}\Deltat _0(s,x; r,z) \Deltat _k(r,z;t,y)$, $n\ge 1\}$
	is uniformly integrable and converges to $\Deltat _0(s,x; \cdot) \Phit(\cdot;t,y)$. 
	This implies  \eqref{eq: sol eq Phi}.
	\end{proof}
 
	Recall that
	$$
		\Phi (s,x;t,y)
		~:=~
		\Phit(s,\AMp_{s}x; t, \AMp_{t}y).
	$$

	\begin{Proposition}\label{prop:existence_fb_regul}  
		Let the conditions of Theorem \ref{thm:f_well_defined}.(i) hold.
		Then, $f: \Theta \longrightarrow \R$ is well-defined in \eqref{eq:def_f_transition_proba}.
		Moreover, it is continuous on $\Theta$ and, for some $C_{[\ref{prop:existence_fb_regul}]}>0$, 
		\begin{align*}
			\big | f(s,x ;t,y) \big|
			~\le~
			C_{[\ref{prop:existence_fb_regul}]}f^{\circ}(s,x;t,y),
			~\mbox{for all}~
			(s,x,t,y)\in \Theta.
		\end{align*}
		
	\end{Proposition}
	\begin{proof}
		This is an immediate consequence of Proposition \ref{prop:existence_continuite_Phi} and Lemma \ref{lem: control density}, 
		recalling that $f^{\circ}$ is a transition density and observing that $\int_{s}^{t}(t-r)^{-1+\kappa_{0}}dr\le C T^{\kappa_{0}}$. 
	\end{proof}

 \subsection{$C^1$-regularity}

	We now prove that $x = (x_1, x_2) \mapsto f(s,x; t, y)$ is $C^{1}$ in its first space variable $x_{1}$, 
	with partial derivative dominated by a Gaussian density.

	\begin{Lemma}\label{lem : borne derive fbar}  
	Let the conditions of Theorem \ref{thm:f_well_defined} hold.
	Then, there exists $C_{[\ref{lem : borne derive fbar}]}>0$ such that, for all $(r,z) \in [0,T] \x \R^2$ and $(s,x,t,y)\in \Theta$,
	\begin{align*}
		\big| \partial_{x_{1}} f_{r, z}(s,x;t,y) \big|
		&~\le~
		\frac{C_{[\ref{lem : borne derive fbar}]}} {(t-s)^{\frac{\beta_0+1}{2}}} f^{\circ}(s,x;t,y).		
 	\end{align*}
	Moreover, let $h: \R^2 \longrightarrow \R$ be a (measurable) function such that $\int_{\R^2} f^{\circ}(s,x; t, y) |h(y)|  dy < \infty$, and 
	$$
		V(s, x; t)  
		~:=~
		\int_{\R^2} f_{t, y}(s,x; t, y) h(y) dy,
		~~(s,x) \in [0,t) \x \R^{2},
	$$
	then   $(s,x)\in [0,t)\x \R^{2}\mapsto V(s,x;t)$ is continuously differentiable in its first space variable $x_{1}$ and  satisfies
	$$
		\big| \partial_{x_{1}} V(s,x;t) \big|
		~\le~
		\frac{C_{[\ref{lem : borne derive fbar}]}} {(t-s)^{\frac{\beta_0+1}{2}}}\int_{\R^{2}}  f^{\circ}(s, x;t,y) ~ |h(y)| dy,
	$$
	in which $C_{[\ref{lem : borne derive fbar}]}>0$ does not depend on $(s,x,t)\in [0,T]\x \R^{2}\x [0,T]$ with $s<t$.
	\end{Lemma}
	\begin{proof} 
		The first inequality follows immediately from Lemmas    \ref{lemma:Sigamw} and \ref{lem: control density}, as in the proof of \eqref{eq: borne I1 estimation L-L}. The second one then follows by dominated convergence. 
	\end{proof}

	For the following, we recall the defintion of $\kappa_{1}$ in  \eqref{eq: kappa1}.

 	\begin{Proposition}\label{prop : fb C1} 
		Let the conditions of Theorem \ref{thm:f_well_defined} hold.
		Then,  for each $(t,y)\in (0,T]\x \R^{2}$, the map 
		$(s,x)\in [0,t)\x \R^{2}\mapsto f(s,x;t,y)$ is continuously differentiable in its first space variable $x_{1}$. Moreover, there exists $C_{[\ref{prop : fb C1}]}>0$ such that 
		$$
			|\partial_{x_{1}}f(s,x;t,y) |
			~\le~
			\frac{C_{[\ref{prop : fb C1}]}}{(t-s)^{1-\kappa_{1}}} f^{\circ}(s,x;t,y),
			~
			\mbox{for all}~
			(s,x;t,y)\in\Theta.
		$$
	\end{Proposition}
	\begin{proof} Fix $z\in \R^{2}$. 
		In view of the estimate in \eqref{eq: estime vraie densite}, together with Lemma \ref{lem : borne derive fbar}, we can find $C>0$, that does not depend on $(s,x,t,y)\in \Theta$, such that   
		\begin{align*}
			&\int_{\R^{2}}\Big| \partial_{x_1}   f_{r,z}(s,x;r,  z) \Phi(r,  z; t,y)\Big|d  z \\
			\le~&
			C (r-s)^{\frac{-\beta_0-1}{2}} \int_{\R^{2}} f^{\circ}(s, x; r,  z) \big|  \Phi(r,  z; t, y) \big| d  z \\
			\le~&
			C (t-r)^{-1 + \kappa_{0}}  (r-s)^{\frac{-\beta_0-1}{2}} \int_{\R^{2}} f^{\circ}(s, x; r,  z) f^{\circ}(r,  z; t,y) d z\\
			=~&
			C (t-r)^{-1 + \kappa_{0}}  (r-s)^{\frac{-\beta_0-1}{2}} f^{\circ}(s, x; t,  y).
		\end{align*}
		Therefore, by the dominated convergence theorem,
		$$
			\partial_{x_1} \int_{s}^{t}\int_{\R^{2}}  f_{r, z}(s,x;r,  z)  \Phi(r,  z,t,y)d  zdr
		$$
		is well-defined and continuous, and 
		so is $\partial_{x_1} f(\cdot; t,y)$. The latter is  bounded from the above estimates by integrating over $r$ and using the relation between the Euler-Gamma  and the Beta functions.  
	\end{proof}

	We conclude this section by a continuity property result on $f$, which allows one to apply the $C^1$-It\^o's formula in the context of Theorem \ref{thm:vr_uniqueX}.
 
	\begin{Proposition}\label{prop : holder f en x2} 
		{Let Assumptions \ref{hyp:A} and \ref{hyp: stand ass}.(i) hold true.
		Assume in addition that \eqref{eq:def_kappa0} holds and that $\frac{\beta_3-\beta_{2}+\beta_{0}}{2} < 1$, 
		and let us fix $\alpha'\in \big( 0,\frac{1+\beta'_{2}}{2} \wedge (1- \frac{\beta_3-\beta_{2}+\beta_{0}}{2}) \big]$.}
		Then, for all $\delta>0$, there exists  $C_{[\ref{prop : holder f en x2}]}>0$ such that  
		\begin{align*}
			\big| f(s,x; t,y) - f(s,x'; t,y) \big|
			&~\le~
			C_{[\ref{prop : holder f en x2}]}|x_{2}-x'_{2}|^{\frac{2\alpha'}{1+\beta'_{2}}},
		\end{align*}
		for all $(s,x,t,y)\in \Theta$ and $x' = (x'_1, x'_2) \in \R^{2}$ such that $t-s\ge \delta$ and $x_{1}=x'_{1}$.
	\end{Proposition}
	\begin{proof} 
		Let $I:=\big| f_{r,z}(s,x; t,y) - f_{r,z}(s,x'; t,y) \big|$ and denote by $C>0$ a generic constant that can change from line to line but does not depend on $(s,x,x',t,y,z)$. 
		Then, by  \eqref{eq:Dx1fb}, Lemma \ref{lemma:Sigamw} and Lemma \ref{lem: control density}, 
		one can find $x''_{2}$ in the interval formed by $x_{2}$ and $x'_{2}$ such that, with $x'':=(x_{1}, x''_{2})$,  
		\begin{align*}
			I&\le |x_{2}-x'_{2}|\;|\partial_{x_{2}}f_{r, z}(s,x''; t,y)|\le  C  |x_{2}-x'_{2}| \frac{1} {(t-s)^{\frac{1+\beta_3}{2}}} f^{\circ}(s,x'';t,y).
		\end{align*}
		If $(t-s)^{\frac{1+\beta'_{2}}{2}}/|x_{2}-x'_{2}|\ge 1$,  then
		\begin{align*}
			I& \le  C |x_{2}-x'_{2}|^{\frac{2\alpha'}{1+\beta'_{2}}} \frac{1}{(t-s)^{\frac{1+\beta_3}{2}-\frac{1+\beta'_{2}}{2}+\alpha'}} f^{\circ}(s,x'';t,y)\\
			&=C |x_{2}-x'_{2}|^{\frac{2\alpha'}{1+\beta'_{2}}} \frac{1}{(t-s)^{\alpha'+\frac{\beta_3-\beta_{2}+\beta_{0}}{2}}} f^{\circ}(s,x'';t,y).
		\end{align*}
		Otherwise, by \eqref{eq: majo fy}, 
		$$
			I\le  |x_{2}-x'_{2}|^{\frac{2\alpha'}{1+\beta'_{2}}} \frac{1}{(t-s)^{\alpha'}} \big( f^{\circ}(s,x; t,y) + f^{\circ}(s,x'; t,y) \big).
		$$
		We conclude by using the fact that $\beta_3-\beta_{2}+\beta_{0}\ge 0$ and by appealing to \eqref{eq: estime vraie densite}.
	\end{proof}

\subsection{$C^2$-regularity} 

	We now prove that  $f$ is $C^{2}$ in its first space variable $x_{1}$ and that $\vr$ is a smooth solution of  the path-dependent PDE \eqref{eq:ppde}.

\subsubsection{Potential estimate and H\"older regularity of $\Phi$}

	Let $0 \le s < t \le T$ and $x \in \R^2$, 
	$h: \R^2 \longrightarrow \R$ be a (measurable) function, we first estimate the second order derivative of the following functional: 
	$$
		V(s,x;t)
		~:=
		\int_{\R^{2}}  f_{t, y}(s,x;t,y) h(y)dy.
	$$
	Let us also denote
  		\begin{align}\label{eq: def bar E-1} 
			\Eb^{-1}_{s,t}(x) 
			~:=~
			\left( \begin{array}{cc} 
			1 & 0 \\
			A_t - A_s & 1
			\end{array} \right)
			x .
		\end{align}
 
	\begin{Lemma}\label{lem : borne derive seconde fbar} 
		Let Assumption \ref{hyp:A} and Assumption \ref{hyp: stand ass}.(i) hold.
		Let $h:  \R^2 \longrightarrow \R$ and $h_{\circ}: \R^2 \longrightarrow \R_+$ be such that,
		for some $\alpha_h > 0$ and $C_h > 0$,
		\begin{align*}
			&\big| h(y) - h(y') \big| 
			\le
			C_{h} \Big( |y_1 - y'_1|^{\frac{2 \alpha_h}{1+\beta'_1}} + |y_2 - y'_2|^{\frac{2 \alpha_h}{1+\beta'_2}} \Big) \big( h_{\circ} (y) + h_{\circ} (y') \big),
			\mbox{ for all } y,y'\in \R^{2},
		\end{align*}
		and
		$$
			\int_{\R^{2}} f^{\circ}(s,x;t,y)  h_{\circ} (y) dy ~<~\infty,
			~\mbox{for all}~
			0 \le s < t \le T,
			~~x \in \R^2.
		$$
		Assume that 
		$$
			\kappa_h
			~:=~ 
			\min \Big(\frac{2\beta_{4}+1+\beta'_{1}}{1+\beta'_{2}}, 1 \Big)\min\{ \alpha_h,\alpha\}- \beta_0
			~>~
			0.
		$$
		Then, $\partial^2_{x_1x_1} V(s,x;t)$ is well defined and continuous.
		Moreover
		\begin{align*} 
			\partial^{2}_{x_{1}x_{1}}V(s,x;t)  
			&=
			\int_{\R^{2}} \partial^{2}_{x_{1}x_{1}} f_{t,y}(s,  x;t,y) h(y)dy,
		\end{align*}
		and there exists $C>0$, that does not depend on $C_{h}>0$, such that 
		\begin{align*} 
			\big| \partial^{2}_{x_{1}x_{1}}V(s,x;t) \big|
			&\le 
			\frac{C C_{h}}{(t-s)^{1-\kappa_{h}}} 
			\left( \big| h \big(\Eb^{-1}_{s,t}(x)\big) \big|
				+
				{
				\big| h_{\circ} \big(\Eb^{-1}_{s,t}(x)\big) \big|
				+
				\int_{\R^{2}} f^{\circ}(s,x;t,y) h_{\circ} (y) dy
				}
			\right),
		\end{align*}
		for all $0 \le s < t \le T$ and $x \in \R^2$.
	\end{Lemma}
	\begin{proof} For simplicity, we only consider the case $t-s\le 1$. 
		To estimate the second order derivative, we  decompose 
		$$
			I:=\int_{\R^{2}}  \partial_{x_{1}x_{1}}^{2}f_{t, y}(s,x;t,y) h(y)dy
		$$
		into the sum of the three following terms, with $\check x:=\Eb^{-1}_{s,t}(x)$, 
		\begin{align*}
			I_{1}&:=\int_{\R^{2}}  \partial_{x_{1}x_{1}}^{2} f_{t, y}(s,x;t,y) \big[ h(y)- h \big(\check x\big) \big]dy, 
			\\
			I_{2}&:= h \big(\check x\big) \int_{\R^{2}}  \left\{\partial_{x_{1}x_{1}}^{2} f_{t,y}(s,x;t,y) - \partial_{x_{1}x_{1}}^{2} f_{t, \check x} (s,x;t,y)\right\} dy,
			\\
			I_{3}&:= h \big(\check x\big) \int_{\R^{2}}  \partial_{x_{1}x_{1}}^{2}f_{t, \check x} (s,x;t,y) dy.
		\end{align*}
	All over this proof, $C>0$ denotes a generic constant that may change from line to line but does not depend on $C_{h}$, $(s,x;t,y)\in \Theta$ and $z\in \R^{2}$.
 
	\vspace{0.5em}

	$\mathrm{(i)}$ We first estimate $I_{1}$. Set $w=\wr_{s,t}(x,y)$, recall \eqref{eq:def_w_xy}.
	By the H\"older regularity property of $h$ and the inequality $(a+b)^{\gamma}\le 2^{\gamma}(a^{\gamma}+b^{\gamma})$ for $a,b\ge 0$ and $\gamma>0$,  one has
	\begin{align*}
		&\Big| \partial_{x_{1}x_{1}}^{2} f_{t, y}(s,x;t,y) \big[h(y)- h \big(\check x\big)] \Big| \\
		&\le
		C_{h} \Big| \partial_{x_{1}x_{1}}^{2}f_{t,y}(s,x;t,y)\Big| \Big( |x_1 - y_1|^{\frac{2\alpha_h}{1+\beta'_1}}+ |y_2 - x_2 -(A_t - A_s) x_1 |^{\frac{2\alpha_h}{1+\beta'_2}} \Big) \big(  h_{\circ} (y) + h_{\circ} (\Eb^{-1}_{s,t}(x))\big) \\
		&\le
		C C_{h}\Big| \partial_{x_{1}x_{1}}^{2}f_{t,y}(s,x;t,y)\Big| \Big( |w_1|^{\frac{2\alpha_h}{1+\beta'_1}} + |w_2|^{\frac{2 \alpha_h}{1+\beta'_2}} +\big |(A_t-A_s) w_1 \big|^{\frac{2\alpha_h}{1+\beta'_2}} \Big) \big( h_{\circ} (y) +h_{\circ} (\Eb^{-1}_{s,t}(x))   \big).
	\end{align*}
	Then, arguing as in the proof of  Lemma \ref{lem : estime L-Lfy} and using \eqref{eq:holder A}, we deduce that 
	\begin{align*} 
		\big| I_1 \big|
		\le
		&C C_{h}\Big( \frac{1}{(t-s)^{1 + \beta_0 - \alpha_h}} + \frac{1}{(t-s)^{1+\beta_{0}-\alpha_{h}\frac{2\beta_{4}+1+\beta'_{1}}{1+\beta'_{2}}}} \Big)
		\int_{\R^{2}} f^{\circ}(s, x;t,y) \big(h_{\circ} (y) + h_{\circ} (\Eb^{-1}_{s,t}(x))  \big) dy,
		\\
		\le~&
		\frac{C C_{h}}{(t-s)^{1-\kappa_{h}}} 
		{\Big( \int_{\R^{2}} f^{\circ}(s,x;t,y) h_{\circ} (y)  dy + h_{\circ} \big(\Eb^{-1}_{s,t}(x) \big) \Big) .}
	\end{align*}
	
	\noindent $\mathrm{(ii)}$ We now consider $I_{2}$. 
	By \eqref{eq:Dx1x2fb} and Lemma \ref{lemma:Sigamw}, 
	\begin{align*}
		& 
		\big|\partial_{x_{1}x_{1}}^{2}f_{t,y}(s,x;t,y)-\partial_{x_{1}x_{1}}^{2}f_{t,\check x}(s,x;t,y) \big| \\
		~\le~&
		\big|  f_{t,y}(s,x; t,y) - f_{t,\check x}(s,x; t,y) \big| \Big| \big(\Sigmab^{-1}_{s,t}(t,y)w \big)_1^2 - \big( \Sigmab^{-1}_{s,t}(t,y) \big)_{1,1} \Big| \\
		&+
		   f_{t,\check x}(s,x; t,y)  \Big| \big(\Sigmab^{-1}_{s,t}(t,y)w \big)_1^2 - \big(\Sigmab^{-1}_{s,t}(t,\check x)w \big)_1^2 \Big| \\
		&+ 
		   f_{t,\check x}(s,x; t,y)  \Big| \big( \Sigmab^{-1}_{s,t}(t,y) \big)_{1,1} -  \big( \Sigmab^{-1}_{s,t}(t,\check x) \big)_{1,1} \Big|\\
		~=~&
		\big|  f_{t,y}(s,x; t,y) - f_{t,\check x}(s,x; t,y) \big| \Big| \big(\Sigmab^{-1}_{s,t}(t,y)w \big)_1^2 - \big( \Sigmab^{-1}_{s,t}(t,y) \big)_{1,1} \Big| \\
		&+
		   f_{t,\check x}(s,x; t,y) \left|\sigma_{t}(y)^{-4}-\sigma_{t}(\check x)^{-4} \right| \Big| \big(\Sigmab^{-1}_{s,t}(1)w \big)_1\Big|^{2} \\
		&+ 
		  f_{t,\check x}(s,x; t,y)  \left|\sigma_{t}(y)^{-2}-\sigma_{t}(\check x)^{-2} \right| \Big| \big( \Sigmab^{-1}_{s,t}(1) \big)_{1,1}  \Big|, 
	\end{align*}
	in which, by \eqref{eq : hyp holder coeff bar sigma},
	\begin{align*} 
		\left|\sigma_{t}(y)-\sigma_{t}(\check x) \right|
		&\le 
		C_{(\ref{eq : hyp holder coeff bar sigma})} 
			\Big(
			\big|  y_{1}-x_{1} \big|^{\frac{2\alpha}{1+\beta'_1}}
			+ 
			 \big|  y_{2}-x_{1}(A_{t}-A_{s})-x_{2}\big|^{\frac{2 \alpha}{1+\beta'_2}} 
			\Big)\\
		&\le C  \Big( |w_1|^{\frac{2\alpha_g}{1+\beta'_1}} + |w_2|^{\frac{2 \alpha_g}{1+\beta'_2}} +\big |(A_t-A_s) w_1 \big|^{\frac{2\alpha_g}{1+\beta'_2}} \Big)
	\end{align*}
	by the same arguments as in in step 1.
	Using Lemma \ref{lem : derive f par rapport vol}  below, \eqref{eq: norne Sigma-1 w 1}, \eqref{eq: borne Sigma -1 11} and  \eqref{eq : unif elliptic}, it follows that
	$$
		\big| I_2 \big|
		~\le~
		 \frac{C |h \big(\check x\big)|}{(t-s)^{1+\beta_{0}-\alpha\frac{2\beta_{4}+1+\beta'_{1}}{1+\beta'_{2}}}} 
		\int_{\R^{2}} f^{\circ}(s, x;t,y) dy\le \frac{C |h \big(\Eb^{-1}_{s,t}(x)\big)|}{(t-s)^{1-\kappa_{h}}}.
	$$

	\noindent $\mathrm{(iii)}$ We finally consider $I_{3}$. 
	Notice that $y \mapsto f_{t,\check x}(s,x; t, y)$ is a Gaussian density function, so that
	$$
		\int_{\R^2} \big| \partial^2_{x_1 x_1} f_{t,\check x}(s,x; t, y) \big| dy 
		~<~
		\infty.
	$$
	Moreover, by the definition of $f_{t,\check x}(s,x; t, y)$, one has
	$$
		D_y f_{t,\check x}(s,x; t, y) 
		= 
		- \left( \begin{array}{cc} 
		1 & 0 \\
		- (A_t - A_s) & 1
		\end{array} \right)^{\top}
		D_{x} f_{t,\check x}(s,x; t,y),
	$$
	so that
	\begin{align*}
		\partial_{x_{1}}f_{t,\check x}(s,x;t,y) 
		=
		-\partial_{y_1} f_{t,\check x}(s,x; t, y)
		- (A_t- A_s) \partial_{y_2} f_{t,\check x}(s,x; t,y),
	\end{align*}
	which implies 
	$$
		\int_{\R} \int_{\R}  \partial_{x_{1}} f_{t,\check x} (s,x;t,y) dy_1 dy_2
		~=~
		0,
	$$
	and therefore $I_3 = 0$.
	
	\vspace{0.5em}
	
	Finally, we can apply the Leibniz integral rule to interchange the derivative and the integral, and hence to conclude the proof. 
	\end{proof}
	
	\begin{Remark} \label{rem:dt_Vsxr}
		Let us consider $V(s,x;t)$ as a path-dependent functional:
		$$
			\overline V(s, \xr; t) 
			~:=~
			V(s, \xr(s), I_s(\xr); t) 
			~=~
			\int_{\R^2} \fr_{t, y}(s, \xr; t, y) h(y) dy. 
		$$
		In view of Remark \ref{rem:PPDE_f_rz},
		the above results implies that, in the context of Lemma \ref{lem : borne derive seconde fbar}, the second order vertical derivative $\partial^2_{\xr \xr} \overline V(s,\xr; t)$ is well defined.
		Moreover, by \eqref{eq:PPDE_f_rz}, one has 
		\begin{equation} \label{eq:bound_ds_fr}
			\partial_s \fr_{t,y}(s, \xr; t,y) 
			~=~
			- \frac12 \sigma(t,y)^{2} \partial^2_{\xr \xr} \fr_{t,y} (s, \xr; t,y).
		\end{equation}
		Then, by the same technique as in Lemma \ref{lem : borne derive seconde fbar}, we can deduce that the horizontal derivative $\partial_s \overline V(s, \xr; t)$ is also well defined.
	\end{Remark}

	We now provide the following easy estimate which is used in the proof of Lemma \ref{lem : borne derive seconde fbar}.

	\begin{Lemma}\label{lem : derive f par rapport vol} 
		Let Assumption \ref{hyp:A}.(i) and Assumption \ref{hyp: stand ass}.(i) hold.
		Then, there exists $C_{[\ref{lem : derive f par rapport vol}]}>0$ such that,    for all $(s,x,t,y)\in \Theta$ and $z,z'\in \R^{2}$, 
		\begin{align*}
		& \big| f_{t,z}(s,x; t,y) - f_{t, z'}(s,x; t,y) \big|\\
		&\le C_{[\ref{lem : derive f par rapport vol}]}\left(|z_{1}-z_{1}'|^{\frac{2\alpha}{1+\beta'_{1}}}+|z_{2}-z_{2}'|^{\frac{2\alpha}{1+\beta'_{2}}}\right)
			\left[1+  \langle  \Sigmab^{-1}_{s,t}(1) w,w\rangle\right]( \varpi f^{\circ})(s,x;t,y) 
		\end{align*}
		in which $w:=\wr_{s,t}(x,y)$.
	\end{Lemma}
  
	\begin{proof} Let us write $f_{[a]}$ for $f_{t, z}$ if $\sigma^{2}_{t}(z)=a$, and let $\partial_{a} f_{[a]}$ denote its derivative with respect to this parameter $a$. Then, 
  $$
  \partial_{a} f_{[a]}(s,x; t,y)=\left[-\frac1{a}+\frac{\sigma_{t}(0)^{2}}{2a^{2}}\langle \Sigmab^{-1}_{s,t}(0)w,w\rangle\right]f_{[a]}(s,x; t,y)
  $$
  in which $w=\wr_{s,t}(x,y)$ is as in \eqref{eq:def_w_xy}. In view of Assumption \ref{hyp: stand ass} and Lemma \ref{lem: control density}, it follows that 
   $$
  |\partial_{a} f_{[a]}(s,x; t,y)|\le C\left[1+\langle \Sigmab^{-1}_{s,t}(0)w,w\rangle\right] \varpi(s,x;t,y) f^{\circ} (s,x;t,y),
  $$
  for some $C>0$ that does not depend on $a, (s,x,t,y)$.  
  We conclude by appealing to \eqref{eq : hyp holder coeff bar sigma}.
    \end{proof}

	In order to apply Lemma \ref{lem : borne derive seconde fbar} to \eqref{eq: def ft recursif}, we need to prove that the function $\Phi(s,x; t, y)$ defined by \eqref{eq:def_Phit} and \eqref{eq:def_Phi} is H\"older in $x$. Recall the definition of $
\widehat{\Delta \beta}$ in \eqref{eq: def wide hat Delat beta}, of $\kappa_{0}$ in \eqref{eq:def_kappa0} and of $f^{\circ,\frac12}$ in \eqref{eq: def fcirc1/2}.

	\begin{Lemma}\label{lem: PHI holder} 
		Let the conditions of Theorem \ref{thm:c12} hold.
		Fix  $\alpha_\Phi \in  \big(0,\hat \alpha_\Phi\wedge \kappa_{0}\wedge \min\limits_{i=1,2} \frac{1+\beta'_{i}}{2} \big)$.
		Then, there exists $C_{\alpha_\Phi}>0$ such that,
		\begin{align*}
			&| \Phi(s,x;t,y)- \Phi(s,x';t,y)| \\
			\le~&
			C_{\alpha_\Phi}
			\frac{|x_1-x_1'|^{\frac{2\alpha_\Phi}{1+\beta'_1}} + |x_2 - x_2'|^{\frac{2\alpha_\Phi}{1+\beta'_2}}}{(t-s)^{1 - \eta_{\Phi} }}
			\Big( f^{\circ,\frac12}(s, x;t,y)+f^{\circ,\frac12}(s,x';t,y) \Big),
		\end{align*}
		for all $(s,x,t,y)\in \Theta$, 
		in which
		$$
			\eta_{\Phi}
			~:=~
			\hat \alpha_\Phi\wedge \kappa_{0}- \alpha_{\Phi}
			~>~
			0.
		$$ 
	\end{Lemma}
	\begin{proof} In all this proof, $C>0$ denotes a generic constant, whose value can change from line to line, but which does not depend on $(s,x,t,y)\in \Theta$. 
	We set 
	$\Delta_{k} (s,x;t,y)
		:=
		\Deltat_{k} (s,\AMp_{s}x;t,\AMp_{t}y)$ and recall that $\Phi (s,x;t,y)
		:=
		\Phit (s,\AMp_{s}x;t,\AMp_{t}y)$, $(s,x,t,y)\in \Theta$. 

	\vspace{0.5em}

	\noindent $\mathrm{(i)}$ Let us first consider 
		\begin{align}
			I
			~:=~ &
			\Delta_0(s,x; t, y) - \Delta_0(s,x'; t, y)  
			\nonumber \\
			=~&
			\mu_{s}(x) \partial_{x_1} f_{t,y} (s,x; t, y) + \frac12 \big( \sigma^2_{s}(x) - \sigma^2_{t}(y) \big) \partial^2_{x_1 x_1} f_{t,y}(s,x; t,y) \nonumber\\
			&-
			\Big(\mu_{s}(x') \partial_{x_1} f_{t,y}(s,x'; t, y) + \frac12 \big( \sigma^2_{s}(x') - \sigma^2_{t}(y) \big) \partial^2_{x_1 x_1} f_{t,y}(s,x'; t,y) \Big). \label{eq: proof C2 def I}
		\end{align}

	\noindent $\mathrm{(i.1)}$ In the case where
	$$
		|x_1 - x_1'|^{\frac{1}{1+\beta'_1}} + |x_2 - x_2'|^{\frac{1}{1+\beta'_2}} 
		~>~
		(t-s)^{1/2},
	$$
	  Lemma \ref{lem : estime L-Lfy} implies that, for $\alpha'\in (0,\kappa_{0})$,
	\begin{align*}
		&\Big| \Delta_0(s, x; t, y) - \Delta_0(s,x'; t, y) \Big|
		\\
		\le~&
		C_{\eqref{eq: estimee L-Lfy}} \frac{1}{(t-s)^{1-\kappa_0}} \big( f^{\circ}(s,x;t,y)+f^{\circ}(s,x';t,y) \big) \\
		\le~&
		C_{\eqref{eq: estimee L-Lfy}}  \frac{|x_1- x_1'|^{\frac{2\alpha'}{1+\beta'_1}} + |x_2 - x_2'|^{\frac{2\alpha'}{1+\beta'_2}} }{(t-s)^{1 - \kappa_0+\alpha'}}
		\big( f^{\circ}(s,x;t,y)+f^{\circ}(s,x';t,y) \big).
	\end{align*}
		
	\noindent $\mathrm{(i.2)}$ We next consider the case where
	\begin{equation} \label{eq:dx_le_dt}
		|x_1 -x_1'|^{\frac{1}{1+\beta'_1}} + |x_2 -x_2'|^{\frac{1}{1+\beta'_2}} 
		~\le~
		(t-s)^{1/2}.
	\end{equation}
	Let us write
	$$
		I 
		~:=~
		\Delta_0(s,x; t,  y) - \Delta_0(s,x'; t, y) 
		~=~
		I_1 + I_2 + I_3 + I_4,
	$$
	where 
	$$
		I_1 ~:=~ \big( \mu_{s}(x) - \mu_{s}(x') \big) \partial_{x_1} f_{t,y}(s,x; t, y),
	$$
	$$
		I_2 ~:=~  \mu_{s}(x')  \big(  \partial_{x_1} f_{t,y}(s,x; t, y )-  \partial_{x_1} f_{t,y}(s,x'; t, y) \big),
	$$
	$$
		I_3 ~:=~ \frac12 \big( \sigma^2_{s}(x) -  \sigma^2_{s}(x') \big) \partial^2_{x_1 x_1} f_{t,y}(s,x'; t, y),
	$$
	and
	$$
		I_4 ~:=~ \frac12 \big( \sigma^2_{s}( x) - \sigma^2_{t}(y) \big) \Big( \partial^2_{x_1 x_1} f_{t,y}(s,x; t, y) - \partial^2_{x_1 x_1} f_{t,y}(s,x'; t, y) \Big).
	$$
	
	For $I_1$, we use the H\"older continuity property of $\mu$ in \eqref{eq : hyp holder coeff bar mu}, Lemma \ref{lemma:Sigamw} and Lemma  \ref{lem: control density}     to obtain that
	\begin{align}\label{eq: proof C2 borne I1}
		\big| I_1 \big|
		&~\le~
		C\frac{ |x_1 - x_1'|^{\frac{2 \alpha}{1+ \beta'_1}} + | x_2 - x_2 '|^{\frac{2 \alpha}{1+\beta'_2}}}{(t-s)^{\frac{1+\beta_{0}}{2}}} f^{\circ}(s,x; t, y).
	\end{align}

	For $I_2$, let us fix  $\rho \in [0,1]$ and $x'' = \rho x + (1-\rho) x'$ so that, using Assumption \ref{hyp: stand ass}, 
	$$
		\big| I_2 \big| 
		~\le~
		\br\Big(
		\big| \partial^2_{x_1 x_1} f_{t,y}(s,x''; t, y) \big| \big|  x_1 -x_1' \big|
		+
		\big| \partial^2_{x_1 x_2} f_{t,y}(s, x''; t, y) \big| \big| x_2 - x_2' \big| \Big).
	$$
	Using \eqref{eq:Dx1x2fb}, \eqref{eq:Dx1x2fb}, Lemma \ref{lemma:Sigamw},  Lemma \ref{lem: control density}  and the fact that $\beta_{0}\le \beta_{3}$, it follows that
	\begin{align*}
		\big| I_2 \big|
		&\le
		C  
		\Big( 
			\frac{|x_1 - x_1'|}{(t-s)^{1+ \beta_{0}}} 
			+
			 \frac{|x_2 - x_2'|}{(t-s)^{1+ \frac{\beta_0+  \beta_{3}}{2}}}  \Big) f^{\circ}(s,x''; t, y).
	\end{align*}
	Since $x''$ lies in the interval formed by $x$ and $x'$,  Lemma \ref{lem: fcirc 1/2} and \eqref{eq:dx_le_dt} imply that 
	\begin{align}\label{eq: proof C2 borne I2}
		\big| I_2 \big|
		&\le
		C  
		\Big( 
			\frac{|x_1 - x_1'|}{(t-s)^{1+ \beta_{0}}} 
			+
			 \frac{|x_2 - x_2'|}{(t-s)^{1+ \frac{\beta_0+  \beta_{3}}{2}}}  \Big) f^{\circ,\frac12}(s,x; t, y).
	\end{align}
	Next, using the H\"older property of $\sigma$ in \eqref{eq : hyp holder coeff bar sigma}, Lemma \ref{lemma:Sigamw} and Lemma \ref{lem: control density}, it follows that
	\begin{align}\label{eq: proof C2 borne I3}
		&\big| I_3 \big|
		~\le~
		\frac{C }{(t-s)^{1+\beta_0}}
		\Big(
		|x_1 - x_1'|^{\frac{2\alpha}{1+\beta'_1}}  + |x_2 - x_2'|^{\frac{2 \alpha}{1+\beta'_2}} 
		\Big)f^{\circ}(s, x; t,y).
		\end{align}
	
	Finally, $I_4$ is tackled as $I_{2}$. Namely,  we can find  $\tilde x'' = \tilde \rho x + (1-\tilde \rho) x'$ with $\tilde \rho \in [0,1]$   such that
	\begin{align*}
	&\Big| \partial^2_{x_1 x_1} f_{t,y}(s,x; t, y) - \partial^2_{x_1 x_1} f_{t,y}(s,x'; t, y) \Big|\\
	\le~& \big| \partial^3_{x_1 x_1 x_{1}} f_{t,y}(s,\tilde x''; t, y) \big| \big|  x_1 -x_1' \big|
		+
		\big| \partial^3_{x_1x_{1} x_2} f_{t,y}(s, \tilde x''; t, y) \big| \big| x_2 - x_2' \big|
	\\
	\le~& C\left(\frac{ \big|    x_1 -x_1' \big|}{(t-s)^{\frac32(1+\beta_{0})}}+\frac{\big| x_2 - x_2' \big|}{(t-s)^{\frac32+ \beta_{0} +\frac{\beta_{3}}{2}}}\right)
	(\varpi^{1}f^{\circ})(s, \tilde  x''; t,y),
	 \end{align*}
	 in which we used Lemma \ref{lemma:Sigamw} and Lemma \ref{lem: control density} again.
	Next, we appeal to  \eqref{eq : hyp holder coeff bar sigma} to deduce that 
	\begin{align*}
		\big| \sigma^2_{s}(x) - \sigma^2_{t}(y) \big| 
		&\le
		C_{\eqref{eq : hyp holder coeff bar sigma}} 
		\Big(|t-s|^{\alpha}+ |  w_{1}|^{\frac{2 \alpha}{1+\beta'_1}} + | w_{2}|^{\frac{2 \alpha}{1+\beta'_2}} \Big) .
	 \end{align*}
	Using that $\beta_{3}\ge \beta_{0}$,   the condition \eqref{eq:dx_le_dt} together with Lemma \ref{lem: fcirc 1/2} and  the fact that $\tilde x''$ lies on the interval formed by $x$ and $x'$ implies that 
	\begin{align}
		|I_{4}|
		&~\le~
		C  |t-s|^{\alpha} \left( \frac{ \big|    x_1-x'_{1}   \big|}{(t-s)^{\frac32(1+\beta_{0})}}+\frac{\big| x_2 - x_2' \big|}{(t-s)^{\frac32+ \beta_{0} +\frac{\beta_{3}}{2}}}\right) 	 f^{\circ,\frac12}(s, x; t,y).\label{eq: proof C2 borne I4} 
	\end{align}
	Note that there exists $C>0$, that does not depend on $(s,   x,  t,  y)$ such that 
	$$
		f^{\circ}(  s,    x;   t,  y)
		~\le~
		Cf^{\circ,\frac12}(  s,    x;   t,  y).
	$$
	Thus, combining \eqref{eq: proof C2 borne I1}-\eqref{eq: proof C2 borne I4} and recalling \eqref{eq: proof C2 def I} and \eqref{eq: majo fcirc par fcirc 1/2}  leads to 
	a upper bound for 
	$$
		J:=\frac{|I|}{C \big( f^{\circ,\frac12}(s, x; t,y)+f^{\circ,\frac12}(s, x'; t,y) \big) }.
	$$
	Namely, 
	\begin{align*}
		J~\le~& \frac{ |x_1 - x_1'|^{\frac{2 \alpha}{1+ \beta'_1}} +|  x_2 - x_2 '|^{\frac{2 \alpha}{1+ \beta'_2}}}{(t-s)^{1+\beta_0}}\\
		 &+~ | x_1 - x_1'|\left(\frac1{(t-s)^{1+\beta_{0} }}+\frac{1}{(t-s)^{\frac32(1+\beta_{0})-\alpha}}\right)\\
		 &+~ |x_2 -   x_2'|\left(\frac1{(t-s)^{1+\frac{\beta_{0}+  \beta_{3}}{2}}}+\frac{1}{(t-s)^{\frac32+\beta_{0}+\frac{\beta_{3}}{2}-\alpha}}\right).
	\end{align*}
	We then use that $(t-s)^{\frac{1+\beta'_{i}}{2}}/|x_{i}-x'_{i}|\ge 1$, for $i=1,2$, to deduce that, for $0 < \alpha'  \le \alpha\wedge \min\limits_{i=1,2} \frac{1+\beta'_{i}}{2}$, 
	\begin{align*}
		J
		~\le~& \frac{ |x_1 - x_1'|^{\frac{2 \alpha'}{1+ \beta'_1}} +|  x_2 - x_2 '|^{\frac{2 \alpha'}{1+ \beta'_2}}}{(t-s)^{1+\beta_0+\alpha'-\alpha}}\\
		 &+~ \frac{ | x_1 -  x_1'|^{\frac{2 \alpha'}{1+ \beta'_1}}}{(t-s)^{(1+\beta_{0})\vee(\frac32(1+\beta_{0})-\alpha)-\frac{1+\beta'_{1}}{2}+\alpha'}}+\frac{ | x_2 - x_2'|^{\frac{2 \alpha'}{1+ \beta'_2}} }{(t-s)^{(1+\frac{\beta_{0}+  \beta_{3}}{2})\vee(\frac32+\beta_{0}+\frac{\beta_{3}}{2}-\alpha)-\frac{1+\beta'_{2}}{2}+\alpha'}}.
	\end{align*}
	Since $\beta_{0}\ge \beta_{1}$ and $\beta_{3}\ge \beta_{2}$,
	\begin{align*}
		J&\le  \frac{|x_1 -  x_1'|^{\frac{2 \alpha'}{1+ \beta'_1}}}{(t-s)^{(\frac12+\beta_{0}+\frac{\beta_{0}-\beta_{1}}{2})\vee(1+\frac32\beta_{0}+\frac{\beta_{0}-\beta_{1}}{2}-\alpha)+\alpha'}} +  \frac{|x_2 -   x_2'|^{\frac{2 \alpha'}{1+ \beta'_2}}}{(t-s)^{(\frac12+\beta_{0}+\frac{  \beta_{3}-\beta_{2}}{2})\vee(1+\frac32 \beta_{0}+\frac{\beta_{3}-\beta_{2}}{2}-\alpha)+\alpha'}}.
	\end{align*}

	\noindent $\mathrm{(i.3)}$ We now combine  the results of steps $\mathrm{(i.1)}$ and $\mathrm{(i.2)}$ to deduce that,
	when $\alpha'= \alpha_\Phi \in  (0,\hat \alpha_\Phi\wedge \kappa_{0})$,
	\begin{align*}
		 |I|
		 &~\le~
		 C  \frac{|x_1 -  x_1'|^{\frac{2 \alpha_\Phi}{1+ \beta'_1}}+|x_2 -   x_2'|^{\frac{2 \alpha_\Phi}{1+ \beta'_2}}}{(t-s)^{1-\eta_{\Phi}}}\left(f^{\circ,\frac12}(s, x; t,y)+f^{\circ,\frac12}(s, x'; t,y)\right).
	\end{align*}
 
 $\mathrm{(ii)}$ To conclude, it remains to use an induction argument as in the end of the proof of Proposition \ref{prop:existence_continuite_Phi}.
	\end{proof}

 \subsubsection{Smoothness of the transition density and Feynman-Kac's representation}\label{sec : Smoothness of the transition density and Feynman-Kac representation}

 	Recall that $\fr(s, \xr; t, y)$ is defined in \eqref{eq: def f bar on path}. 
 
	\begin{Proposition} \label{prop:fr_C12}
		{Let the conditions of Theorem \ref{thm:c12} hold.} Then, the vertical derivative $\partial^2_{\xr \xr} \fr(s, \xr; t, y)$ and horizontal derivative $\partial_s \fr(s, \xr; t, y)$ are well-defined for all $0 \le s < t \le T$, $\xr \in D([0,T])$ and $y \in \R^2$.
		Moreover, for all $(t,y) \in [0,T] \x \R^2$,  $\partial^2_{\xr \xr} \fr(\cdot; t, y)$ and $\partial_s \fr(\cdot; t, y)$ are continuous on $[0,t) \x C([0,T])$.
	\end{Proposition} 
	
	\begin{proof} We denote by $C>0$ a generic constant that does not depend on $(s,x, t,y)$.
	Let us fix $t_0 \in (s, t)$, then by \eqref{eq:def_f_transition_proba} and  \eqref{eq: def f bar on path},
	\begin{align*} 
		\fr(s, \xr; t, y)
		~:=~&
		\fr_{t,y}(s, \xr; t,y)
		+
		\int_s^{t_0}\!\! \int_{\R^2} \fr_{r,z}(s, \xr; r,z) \Phi(r,z; t,y) dz dr \\
		&+
		\int_{t_0}^t \int_{\R^{2}} \fr_{r,z}(s, \xr; r,z) \Phi(r,z; t,y) dz dr \\
		~=:~&
		\fr_{t,y}(s, \xr; t,y)
		~+~
		\fr_1(s, \xr; t,y)
		~+~
		\fr_2(s, \xr; t,y).
	\end{align*}

	First, the existence and continuity of the vertical derivative and horizontal derivative of $\fr_{t,y}(s, \xr; t,y)$ is trivial.

	\vspace{0.5em}
	
	For $\fr_1(s, \xr; t,y)$,   we can use  Lemmas \ref{lem : borne derive seconde fbar} and \ref{lem: PHI holder}, {Proposition \ref{prop:existence_continuite_Phi}}, 
	together with \eqref{eq: majo fcirc par fcirc 1/2}, to obtain that
	$$
		\int_s^{t_0}\!\! \left| \int_{\R^2} \partial^2_{\xr \xr } \fr_{r,z}(s, \xr; r,z) \Phi(r,z; t,y) dz \right|  dr  
		~\le~
		C \int_s^{t_0} \frac{I_{1}(s, \xr; r; t,y)  +I_{2}(s, \xr; r; t,y) }{(r-s)^{1-\kappa_{\Phi}} } dr,
	$$ 
	where
	$$
			\kappa_\Phi
			~:=~ 
			\min \Big(\frac{2\beta_{4}+1+\beta'_{1}}{1+\beta'_{2}}, 1 \Big)\min\{ \alpha_\Phi,\alpha\}- \beta_0
			~>~
			0,
	$$
	and, with $x := (\xr(s), I_s(\xr))$,
	\begin{align*}
		I_{1}(s,\xr; r; t,y) ~:=~& \int_{\R^{2}}  f^{\circ,\frac12}(s, x;r,z) \big( f^{\circ,\frac12}(r, \Eb^{-1}_{s,r}(x);t,y)+  f^{\circ,\frac12}(r,z;t,y)  \big) dz
		\\
		~=~&  f^{\circ,\frac12}(r, \Eb^{-1}_{s,r}(x);t,y)+  f^{\circ,\frac12}(s,x;t,y) \\
		I_{2}(s,\xr; r; t,y)
		~:=~&f^{\circ,\frac12}(r,\Eb^{-1}_{s,r}(x);t,y).  
	\end{align*}
	Since $t_0 < t$, we can then easily obtain the existence and continuity of $\partial^2_{\xr \xr} \fr_1(\cdot; t, y)$ by dominated convergence.
	Further, in view of Remark \ref{rem:dt_Vsxr} and in particular \eqref{eq:bound_ds_fr}, 
	we can also deduce the existence and continuity of the horizontal derivative $\partial_s \fr_1(\cdot; t,y)$.
	
	\vspace{0.5em}
	
	For $\fr_2(s, \xr; t, y)$, we notice that 
	$$
		\big| \partial_s \fr_{r,z}(s, \xr; r,z) \big|
		+
		\big| \partial^2_{\xr\xr} \fr_{r,z}(s, \xr; r,z) \big|
		~\le~
		C \fr^{\circ}(s, \xr; r,z), 
		~
		\mbox{for}~r \ge t_0 > s, ~z \in \R^2.
	$$
	Together with the estimate on $\Phi(r,z; t, y)$ in Proposition \ref{prop:existence_continuite_Phi},
	it follows the existence and continuity of the vertical derivative $\partial^2_{\xr\xr} \fr_2(\cdot; t, y)$ and the horizontal derivative $\partial_s \fr_2(\cdot; t,y)$.
	\end{proof}

	{Recall the growth condition \eqref{eq:bound_g_l} on $\ell$ and $g$, and the H\"older continuity condition \eqref{eq:holder_l} on $\ell$.
	Let}
	$$
		v(s,x) ~:=~ \int_s^T\!\! \int_{\R^2} \ell(t,y) f(s, x; t, y) dy dt + \int_{\R^2} g(y) f(s, x; T, y) dy,
		~~(s,x) \in [0,T) \x \R^2.
	$$
	Then, with $\vr$ defined in \eqref{eq:def_vr}, one has, for $x = (\xr(s), I_s(\xr))$,
	$$
		\vr(s,\xr) = v(s,x),
		~
		\partial_{\xr} \vr(s,\xr) = \partial_{x_1}v(s,x)
		~\mbox{and}~
		\partial^2_{\xr\xr} \vr(s,\xr) = \partial^2_{x_1 x_1}v(s,x).
 	$$

	\begin{Proposition}\label{prop:vr_C12_edp}
		{Let the conditions of Theorem \ref{thm:c12} hold. Then:}
		
		\vspace{0.5em}
		
		\noindent $\mathrm{(i)}$  $\vr\in \Cb^{1,2}([0,T))$ and the bound estimates in \eqref{eq:bound_dvr} hold true.
		
		\vspace{0.5em}

		\noindent $\mathrm{(ii)}$ The function $\vr$ is a classical solution to the PPDE \eqref{eq:ppde}.
		If in addition $g$ is continuous, {then $\vr$ is the unique classical solution of  \eqref{eq:ppde} satisfying \eqref{eq: borne croissance vr et cond bord}.}
	\end{Proposition}
	\begin{proof}
	$\mathrm{(i)}$ 
	Let us define, for $(r,z) \in [0,T) \x \R^2$, 
	\begin{align}\label{eq: def v bar Phi} 
		v_{ \Phi}(r,z)
		~:=
		\int_r^T \int_{\R^2} \Phi(r,z; t,y) \ell(t,y) dy dt 
		~+
		\int_{\R^{2}} \Phi(r,z;T,y)g(y)dy ,
	\end{align}
	so that
  	\begin{align} \label{eq : partial x1 x1 v(s,barx)}
		v(s,x) ~=&
		\int_{\R^2} f_{T,y} (s,x; T,y) g(y) dy ~+ \int_s^T \int_{\R^2} f_{r,z}(s,x; r,z) \big( v_{\Phi}(r,z) + \ell(r,z) \big) dz dr.
  	\end{align}
	Then, it follows from Lemma \ref{lem: PHI holder} that 
  	\begin{align*}
		| v_{ \Phi}(r,z)- v_{\Phi}(r,z')|
		\le
		C_{\alpha_\Phi}
			\frac{|z_1-z_1'|^{\frac{2\alpha_\Phi}{1+\beta'_1}} + |z_2 - z_2'|^{\frac{2\alpha_\Phi}{1+\beta'_2}}}{(T-r)^{1 - \eta_{\Phi} }}
			\big(  v^{\circ,\frac12}(r,z)+ v^{\circ,\frac12} (r,z') \big),
  	\end{align*}
	in which
	$$ 
		v^{\circ,\frac12}(r,z)
		~:=~
		\int_r^T \int_{\R^2} f^{\circ,\frac12}(r, z;T,  y) \big| \ell(t,y) \big| dy dt 
		+
		\int_{\R^{2}} f^{\circ,\frac12}(r, z;T,  y) \big|g(y) \big| dy.
	$$
{Together with the H\"older continuity condition on $\ell$ in \eqref{eq:holder_l}, we can then apply Lemma \ref{lem : borne derive seconde fbar}
	to deduce that $\partial^2_{x_1 x_1} v(s,x)$ exists and
	\begin{align*} 
		&~\partial^{2}_{x_{1}x_{1}} v(s,x) \\
		=&
		\int_{\R^{2}} \! \partial^{2}_{x_{1}x_{1}}  f_{T,y}(s,x;T,y)g(y)dy 
		+\!
		\int_{s}^{T} \!\!\!  \int_{\R^{2}} \! \partial^{2}_{x_{1}x_{1}} f_{r,z}(s,x;r,z)  \big( v_{\Phi}(r,z) + \ell(r.z) \big) dzdr.
	\end{align*}

	Then, using \eqref{eq:Dx1x2fb} and Lemma \ref{lemma:Sigamw}, we deduce that, for some constant $C > 0$,
	$$
		\int_{\R^{2}} \! \Big| \partial^{2}_{x_{1}x_{1}}  f_{T,y}(s,x;T,y)g(y) \Big| dy 
		~\le~
		\frac{C}{(T-s)^{1+\beta_0}} \int_{\R^2} f^{\circ}(s, x; T,y) | g(y)| dy
		~\le~
		\frac{C e^{C|x|} }{(T-s)^{1+\beta_0}}.
	$$
	By Lemma \ref{lem : borne derive seconde fbar}, one can choose $C> 0$ such that,  
	\begin{align*}
		&\left| \int_{s}^{T} \!\!\!  \int_{\R^{2}} \! \partial^{2}_{x_{1}x_{1}} f_{r,z}(s,x;r,z)  \ell(r.z) dzdr \right| \\
		\le~&
		\int_s^T \frac{C}{(r-s)^{1-\kappa_{\ell}}} \Big( \big| \ell( \Eb^{-1}_{s,r}(x)) \big| + C e^{| \Eb^{-1}_{s,r}(x)|} + \int_{\R^2} f^{\circ}(s,x; r,z) e^{C|z|} dz \Big) dr
		~\le~
		C e^{C|x|},
	\end{align*}
		{in which
			$$
			\kappa_\ell
			~:=~ 
			\min \Big(\frac{2\beta_{4}+1+\beta'_{1}}{1+\beta'_{2}}, 1 \Big)\min\{ \alpha_\ell,\alpha\}- \beta_0
			~>~
			0,
		$$
		}
	and
	\begin{align*}
		\left| \int_{s}^{T} \!\!\!  \int_{\R^{2}} \! \partial^{2}_{x_{1}x_{1}} f_{r,z}(s,x;r,z)  v_{\Phi} (r.z) dzdr \right|
		\le
		\int_s^T \frac{C}{(r-s)^{1-\kappa_{\Phi}} (T-r)^{1-\eta_{\Phi}} } e^{C|x|} dr
		~\le~
		C e^{C|x|}.
	\end{align*}
	This proves the bound estimate on $\partial^2_{x_1x_1} v(s,x)$ (or equivalently $\partial^2_{\xr \xr} \vr(s, \xr)$) in \eqref{eq:bound_dvr}.
	In view of   \eqref{eq:bound_ds_fr}, one can obtain the same bound   on $\partial_s \vr(s, \xr)$ in  \eqref{eq:bound_dvr}.
	Finally,  $\partial_{\xr} \vr(s, \xr)$ is estimated by appealing to  Proposition \ref{prop : fb C1} and   \eqref{eq:bound_g_l}. {The bound on the right-hand side of \eqref{eq: borne croissance vr et cond bord} is proved similarly.}
}

	\vspace{0.5em}
	
	$\mathrm{(ii)}$ 
	Recall that  
	\begin{align*}
	 	 \fr(s,\xr;t,y) 
		 ~=~
		 \fr_{t, y}(s,\xr;t,y)
		 ~+
		 \int_{s}^{t}\int_{\R^{2}} \fr_{r, z}(s,\xr;r,z) \Phi(r,z,t,y)dzdr,
	\end{align*}
	and that $(s,\xr)\in [0,t)\x D([0,T])\mapsto  \fr_{t,y}(s,\xr;t,y)$ solves 
	\begin{align}\label{eq: LcircE bar fE =0} 
		{\rm L}_{t, y} \fr_{t, y}(\cdot;t,y)=0 \mbox{ on $[s,t)\x C([0,T])$,}
	\end{align}
	where 
	$$
		 {\rm L}_{t, y} ~:=~  \partial_s  + \frac12\sigma_{t}(  y)^{2} \partial^{2}_{\xr \xr}.
 	$$
	Let
	$$
		{\rm L} \phi(s,\xr)  
		~:=~
		 \partial_s \phi(s, \xr) + \mu_{s}(\xr) \partial_{\xr}\phi(s,\xr)+\frac12\sigma_{s}^{2}(\xr)\partial^{2}_{\xr}\phi(s,\xr),
	$$
	for $\phi \in \Cb^{1,2}([0,T))$.
	Recalling the definition of   $v_{ \Phi}$   in \eqref{eq: def v bar Phi} and using \eqref{eq : partial x1 x1 v(s,barx)}, we obtain that,  with $x := (\xr(s), I_s(\xr))$, 
	\begin{align}
		{\rm L}  \vr(s,\xr)
		=& \int_{\R^{2}}{\rm L}  \fr_{ {T},y}(s,\xr;T,y)g(y)dy
		- v_{ \Phi}(s, x) {- \ell(s,x)} \nonumber \\
		&+
		 \int_{s}^{T} \!\! \int_{\R^{2}} {\rm L}  \fr_{r, z}(s,\xr;r,z) \big( v_{ \Phi}(r,z) { + \ell(r,z)} \big) dz dr.
		 \label{eq: Lxr tilde vr}
	\end{align}
	At the same time, as a consequence of \eqref{eq: sol eq Phi} and \eqref{eq: lien derive f bar f}-\eqref{eq: lien derive seconde f bar f}, 
	we observe that
	\begin{align*}
		\Phi(s,x;T,y)
		=~&
		({\rm L} -{\rm L}_{t, y})  \fr_{t, y}(s,\xr; t,y) \\
		&+
		\int_{s}^{t} \! \int_{\R^{2}}({\rm L} -{\rm L}_{r,z})  { \fr_{r,z}}(s,\xr;r,z) \Phi(r,z;t,y) dz dr.
	\end{align*}
	Hence, recalling Lemma \ref{lem : estime L-Lfy} and Proposition \ref{prop:existence_continuite_Phi}, it follows by \eqref{eq: def v bar Phi} that
	\begin{align*}
		v_{ \Phi}(s,x)=&\int_{\R^{2}}({\rm L} -{\rm L}_{T,   y})  \fr_{ {T,}y}(s,\xr;T,  y)g(y)d  y
		\\
		&+\int_{s}^{T} \!\! \int_{\R^{2}}({\rm L} -{\rm L}_{r,z})  \fr_{ {r,}z}(s,\xr;r,z)  \big( v_{ \Phi}(r,z) {+ \ell(r,z) } \big) dz dr.
	\end{align*}
	We then use \eqref{eq: LcircE bar fE =0}  to obtain
	\begin{align*}
		v_{\Phi}(s,x)
		~=&
		\int_{\R^{2}}{\rm L}   \fr_{T, y}(s,\xr;T,  y)g(  y)d  y
		+
		\int_{s}^{T} \!\! \int_{\R^{2}} {\rm L}  \fr_{r, z}(s,\xr;r,z)  \big( v_{\Phi} (r,z) {+ \ell(r,z) } \big) dz dr.
	\end{align*}
	It follows then by \eqref{eq: Lxr tilde vr} that $\vr$ is a classical solution to the PPDE \eqref{eq:ppde}. 

	\vspace{0.5em}
	
	\noindent $\mathrm{(iii)}$ We now prove   that $\lim_{s \nearrow T} \vr(s,\xr) = g(\xr_T, I_T(\xr))$, or equivalently $\lim_{s \nearrow T} v(s,x) = g(x) $, {whenever $g$ is continuous}.
In view of the estimates in \eqref{eq: majo fy} and \eqref{eq: estime vraie densite}, and    Proposition \ref{prop:existence_fb_regul},
	one has
	\begin{align*} 
		\lim_{s \nearrow T} v(s,x)
		~=~&
		\lim_{s \nearrow T} \int_{\R^2} f(s,x; T,y) g(y) dy
		~=~
		\lim_{s \nearrow T} \int_{\R^2} f_{T,y} (s,x; T,y) g(y) dy \\
		~=~&
		\lim_{M \to \infty} \lim_{s \nearrow T} \int_{D^M_{s,T}} \!\! f_{T,y} (s,x; T,y) g(y) dy 
		~=
		\lim_{M \to \infty} \lim_{s \nearrow T} \int_{D^M_{s,T}} \!\! f_{T,x} (s,x; T,y) g(y) dy \\
		~=~&
		\lim_{s \nearrow T} \int_{\R^2} f_{T,x} (s,x; T,y) g(y) dy
		~=~
		g(x),
	\end{align*}
	in which 
	$$
		D^M_{s,T} 
		:=
		\Big[ x_1 -M \sqrt{T-s}, ~x_1+ M \sqrt{T-s} \Big] \x \Big[ x_2 - M \sqrt{(T-s) \tilde m_{s,t}}, ~x_2 + M \sqrt{(T-s) \tilde m_{s,t}} \Big],
	$$
	so that third and fifth equalities are true since both $f_{T,y}(s,x; T,y)$ and $f_{T,x}(s,x; T,y)$ are dominated by $C f^{\circ}(s,x; T,y)$ in which the covariance matrix in $f^{\circ}$ is given by $\Sigma_{s,T}(4\bar \ar)$,
	and the fourth equality follows by the fact that, for every fixed $M > 0$, 
	$$
		\lim_{s \nearrow T} \sup_{y \in D^M_{s,T}} \left| \frac{f_{T,y}(s,x; T,y)}{ f_{T,x}(s,x; T,y)} - 1 \right| = 0.
	$$	

{\noindent $\mathrm{(iv)}$ The fact that $\vr$ is the unique solution of  \eqref{eq:ppde} satisfying \eqref{eq: borne croissance vr et cond bord} holds true follows easily by a verification argument based on It\^{o}-Dupire's formula, see \cite{cont2013functional}, whenever $g$ is continuous. }
	\end{proof}

\subsection{Proofs of Theorems \ref{thm:f_well_defined}, \ref{thm:c12} and \ref{thm:vr_uniqueX}}

\begin{proof}[\sl Proof of Theorem \ref{thm:f_well_defined}]
	$\mathrm{(i)}$ First, the well-posedness of $\Phit$ in \eqref{eq:def_Phit}-\eqref{eq: def Delta L k} is proved in Proposition \ref{prop:existence_continuite_Phi}.
	Further, the well-posedness of $f$ in \eqref{eq:def_f_transition_proba} as well as its continuity and growth property is proved in Proposition \ref{prop:existence_fb_regul}.

	\vspace{0.5em}
	
	\noindent $\mathrm{(ii)}$ Under further conditions, the existence of $\partial_{x_1} f(s,x; t,y)$ as well as its continuity and growth property is proved in Proposition \ref{prop : fb C1}.
\end{proof}

\begin{proof}[\sl Proof of Theorem \ref{thm:c12}]
	$\mathrm{(i)}$ The fact that $\fr(\cdot; t, y) \in \Cb^{1,2}([0,t))$ is proved in Proposition \ref{prop:fr_C12}.
	
	\vspace{0.5em}
	
	\noindent $\mathrm{(ii)}$ The fact that $\vr$ provides a classical solution to the PPDE, as well as the estimation on the derivatives are proved in Proposition \ref{prop:vr_C12_edp}.

	\vspace{0.5em}
	
	\noindent $\mathrm{(iii)}$ 
	We now use the PPDE results in Item $\mathrm{(ii)}$ to study the path-dependent SDE \eqref{eq:SDE_XI}.
	To study the weak solution of the SDE \eqref{eq:SDE_XI}, we consider the martingale problem on the canonical space $C([0,T])$ of all $\R^2$-valued continuous paths on $[0,T]$.
	By abuse of notation, we denote by $(X_t, I_t)_{t \in [0,T]}$ the canonical process, which generates the canonical filtration $\F$.
	Then, given an initial condition $(t, x) \in [0, T] \x \R^2$, a solution to the corresponding martingale problem is a probability measure $\P$ on $C([0,T])$ such that
	$\P[(X_s, I_s) = x = (x_1, x_2), ~s \in [0,t]] = 1$, $\P[I_s = x_2 + \int_t^s X_r dA_r, ~s \in [t, T]] = 1$ and the process
	$$
		\varphi(X_s) - \int_t^s \Big( \mub_r(X) D \varphi(X_r) + \frac12 \sigmab_r (X) D^2 \varphi(X_r) \Big) dr,
		~s \in [t,T],
	$$
	is a $(\P, \F)$-martingale for all bounded smooth functions $\varphi: \R \longrightarrow \R$.
	Let us denote, for all $(t,x) \in [0,T] \x \R^2$,
	$$
		\Pc(t,x) 
		~:=~
		\big\{ \P ~: \P ~\mbox{is solution to the martingale problem with initial condition}~(t,x) \big\}.
	$$
	Notice that $\mub$ and $\sigmab$ are both bounded continuous, 
	it is then classical to know that $\Pc(t,x)$ is a nonempty compact set (see e.g. Stroock and Varadhan \cite[Chapter VI]{stroock1997multidimensional}).
	
	\vspace{0.5em}
	
	We next apply the classical Markovian selection technique (see e.g. \cite[Chapter 12.2]{stroock1997multidimensional}) to construct a weak solution to the SDE \eqref{eq:SDE_XI} such that $(X_t, I_t)_{t \in [0,T]}$ is a strong Markov process.
	Let $(\phi_n)_{n \ge 1}$ be a sequence of bounded continuous functions from $[0,T] \x \R^2 \longrightarrow \R$ such that  it is a measure determining sequence in the sense that
	the sequence
	$$
		\Big\{ \E^{\P} \Big[ \int_0^T \phi_n(t, X_t, I_t) dt \Big] \Big\}_{n \ge 1}
	$$
	can determinate the probability measure $\P$ on $C([0,T])$.
	For each $(t,x) \in [0,T] \x \R^2$, let $\Pc^+_0(t,x) := \Pc(t,x)$, and then define, for each $n \ge 0$,
	$$
		\Pc^+_{n+1}(t,x) 
		=
		\Big\{ 
			\P \in \Pc^+_n (t,x) ~: 
			\E^{\P} \Big[ \int_0^T \phi_n(t, X_t, I_t) dt \Big]  
			=
			\max_{\P' \in \Pc^+_n(t,x)} \E^{\P'} \Big[ \int_0^T \phi_n(t, X_t, I_t) dt \Big]
		\Big\}.
	$$
	It is easy to see that each $\Pc^+_n(t,x)$ is a non-empty compact set, so that $\Pc^+(t,x) := \cap_{n \ge 1} \Pc^+_n(t,x)$ is also non-empty compact, {as the sequence is non-increasing}.
	Moreover, since any two probability measures in $\Pc^+(t,x)$ has the same value by evaluating w.r.t. any $\phi_n$,
	this implies that $\Pc^+(t,x)$ contains exactly one probability measure denoted by $\P^+_{t,x}$.
	By the dynamic programming principle for the optimal control problem in the definition of $\Pc^+_{n+1}$,
	it follows that $(X, I, (\P^+_{t,x})_{(t,x) \in [0,T] \x \R^2})$ provides a Markov process solution to  SDE \eqref{eq:SDE_XI} such that $( X, I)$ is a strong Markov process.
	
	\vspace{0.5em}
	
	At the same time, one can apply the above Markovian selection argument to construct another Markov process $(X, I, (\P^-_{t,x})_{(t,x) \in [0,T] \x \R^2}$ by replacing  ``$\max$'' by ``$\min$'' in the definition of $\Pc^+_{n+1}(t,x)$.
	If the class of all martingale solutions $\Pc(t,x)$ is not unique, then $\P^+_{t,x} \neq \P^-_{t,x}$ as $(\phi_n)_{n \ge 1}$ is measure determining.
	
	\vspace{0.5em}
	
	At the same time, by the results in Item $\mathrm{(ii)}$ and the Feynman-Kac's formula in the case $g\equiv 0$, one has
	$$
		\E^{\P^+_{s,x}} \Big[ \int_s^T \ell_t(X_t, I_t) dt \Big] 
		=
		\E^{\P^-_{s,x}} \Big[ \int_s^T \ell_t(X_t, I_t) dt \Big] 
		=
		\int_s^T \int_{\R^2} f(s,x; t,y) \ell_t(y) dy dt.
	$$
	Since $\ell$ could be an arbitrary bounded continuous function, this implies that $\P^+_{s,x} = \P^-_{s,x}$ for all $(s,x) \in [0,T] \x \R^2$.
	Therefore, for all initial condition $(t,x)$, there exists a unique solution to the martingale problem, i.e. a unique weak solution.
	Moreover the (unique) solution process $(X, I)$ is a strong Markov process,
	and the transition probability function is given by $f$.
\end{proof}

 \begin{proof}[\sl Proof of Theorem  \ref{thm:vr_uniqueX}.]

	When the SDE \eqref{eq:SDE_XI} admits weak uniqueness, the above Markovian selection argument shows that the only solution $( X, I)$ is a strong Markov process.
	
	\vspace{0.5em}
 
	$\mathrm{(i)}$ Let $W^{\perp}$ be a Brownian motion independent of $W$, $(\eps_n)_{n \ge 1}$ be a sequence of positive constants such that $\eps_n \longrightarrow 0$.
	For each $n > 1$, let us define $\Xt^{n} = (\Xt^{n,1}, \Xt^{n,2})$ as the unique (Markovian) solution to the SDE
 	\begin{align*} 
		d \Xt^{n,1}_t &
		=
		\mut_t (\Xt^n_t) dt + \sigmat_t(\Xt^n_t) dW_t,
		~~~
		d \Xt^{n,2}_t 
		=
		\mut_t(\Xt^n_t) A_t dt +   \sigmat_t(\Xt^n_t) A_t d W_t + \eps_n dW^{\perp}_t.
	\end{align*}
	By stability of weak solutions of SDEs, it is clear that, by using the same initial condition for the above SDE as that in \eqref{eq:def_Xt} for $\Xt$, 
	one has $\Xt^n \longrightarrow \Xt$ weakly.
	
	\vspace{0.5em}
	
 	At the same time, it follows from e.g.~\cite{francesco2005class} that, for each $t\in (s,T]$, $\Xt^{n}_{t}$ has a density $\ft^{n}(s,x;t,\cdot)$ whenever $\Xt^{n}_{s}=x$.
	Moreover, $\ft^n$ can be defined in the form 
 	\begin{equation*}  
		\ft^{n}(s,x;t,y)= \ft^{n}_{{t},y}(s,x;t,y)+ \int_{s}^{t}\int_{\R^{2}} \ft^{n}_{{r},z}(s,x;r,z)\Phit^{n}(r,z;t,y)dzdr
	\end{equation*}
	in which 
	$\ft^{n}_{{t},y}$ is defined as $\ft_{{t},y}$ but with 
	 $$
		\Sigmab^{n}_{s,t} (r, z)
		=
		\sigma^2_{r} (z) 
		\left( \begin{array}{cc} 
		t- s &  - \int_s^t (A_u - A_s) du  \\
		- \int_s^t (A_u - A_s) du & \int_s^t \big[(A_u - A_s)^2+ \eps^{2}_{{n}} \sigma_r (z)^{-2} \big] du
		\end{array} \right)
	$$
	in place of $\Sigmab_{s,t} (r, z)$, and 
	\begin{equation*} 
		\Phit^{n}(s,x; t, y) 
		~:=~
		\sum_{k=0}^{\infty} \Deltat^{n}_k(s,x; t,y),
	\end{equation*}
	where $\Deltat^{n} _0(s,x; t,y) := \big( \Lct^n_s - \Lct^{n,t,y}_s \big) \ft^{n}_{t,y}(s,x; t,y) $,
	\begin{align*}
		\Deltat^{n}_{k+1}  (s,x;t,y) := \int_s^t   \int_{\R^2} \Deltat^{n} _0(s,x; r,z) \Deltat^{n}_{k}(r,z;t,y) dz dr, 
		~~k\ge 0.
	\end{align*}
	In the above, $\Lct^n$ is the generator of $\Xt^{n}$ and  $\Lct^{n,t,y}$ is defined from  $\Lct^n$  as $\Lct^{t,y}$ is defined from $\Lct$ by freezing $\sigma$ to $\sigma_{t}(y)$ and erasing the drift term. Then, we define $f^{n}_{r,z}$ from  $\ft^n_{\cdot}$ as  $f_{r,z}$ is defined from $ \ft_{\cdot}$ in Section \ref{sec:main_results}.	
	\vspace{0.5em}
	
	It is straightforward to check that the estimates in \eqref{eq: estime Delta k} hold for $(\Deltat^{n}_{k})_{k\ge 0}$ in place of $(\Deltat_{k})_{k\ge 0}$, uniformly in $n>0$. 
	Then, an induction argument combined with the fact that $\ft^{n}_{t, y}(s,x; t,y)\to \ft_{t, y}(s,x; t,y)$ as $n\to \infty$, for all $(s,x,t,y)\in \Theta$, 
	implies that $f^{n}(s, x;t, y):= \ft^n (s,\AMp_{s} x;t,\AMp_{t}  y)$ converges to $ f(s, x;t, y)$ as $n \longrightarrow \infty$, for all $(s, x,t, y)\in \Theta$. 
	By the weak convergence of {the sequence of processes $(\Xt^n)_{n \ge 1}$ to $X$}, this shows that $f$ is the transition probability function of {$( X, I)$}.

	\vspace{0.5em} 
 
	\noindent $\mathrm{(ii)}$ As shown in Theorem \ref{thm:f_well_defined}, one has $\vr\in \Cb^{0,1}([0,T))$ and the vertical derivative $\partial_{\xr}\vr$ is locally bounded.
	Let $(X,I)$ be the solution of SDE \eqref{eq:SDE_XI}, then by Feynman-Kac's formula, the process
	$$
		\vr(t,X) + \int_0^t \bar \ell(s, X) ds ,~t \in [0,T],
		~\mbox{is a local martingale}.
	$$
	One can further apply the $C^1$-It\^o formula for path-dependent functionals in \cite{bouchard2021c} to prove \eqref{eq: Ito vr}.
	Indeed, when \eqref{eq: cond vr Lip} holds true, one can directly apply  \cite[Proposition 2.11 and Theorem 2.5]{bouchard2021c}.

	\vspace{0.5em}

	Otherwise, when $A$ is monotone and $0<\frac{1+\beta_{2}-\beta_{0}}{2+4\beta_{4}} < 1-\frac{\beta_3-\beta_{2}+\beta_{0}}{2}$,
	we can fix $\alpha'\in (\frac{1+\beta_{2}-\beta_{0}}{2+4\beta_{4}},1-\frac{\beta_3-\beta_{2}+\beta_{0}}{2})$, and, 
	by Proposition \ref{prop : holder f en x2},
	there exists a constant $C> 0$ such that, for all $\eps > 0$,
	\begin{align*}
		&\E \Big[ \Big| \vr\big(s+\eps, X \big)-\vr\big(s+\eps,  X_{s \wedge \cdot} \oplus_{s+\eps} (X_{s+\eps} - X_s) \big) \Big|^{2} \Big]
		\\
		~\le~&
		C \E\Big[ \Big(\sup_{ s \le t \le s+\eps} \big| X_t - X_s \big| ~\eps^{\beta_{4}} \Big)^{\frac{4\alpha'}{1+\beta_{2}'}} \Big]
		~\le~
		C\eps^{\frac{\alpha'(2 +4\beta_{4})}{1+\beta_{2}'}},
	\end{align*}
	where $ \big( X_{s \wedge \cdot} \oplus_{s+\eps} (X_{s+\eps} - X_s) \big)_t  :=  \1_{[0, s+\eps)}(t) X_{s\wedge t} + \1_{[s+\eps,T]} (t) X_{s+\eps}$ for all $t \in [0,T]$.
	Since $\frac{\alpha'(2+4\beta_{4})}{1+\beta_{2}'}>1$, it follows that 
	$$
		\lim_{\eps\searrow 0} 
		\frac1\eps
		\E\Big[ \Big|\vr\big(s+\eps, X\big)-\vr\big( X_{s \wedge \cdot} \oplus_{s+\eps} (X_{s+\eps} - X_s) \big) 
		\Big|^{2} 
		\Big]
		~=~
		0.
	$$
	Finally, we can apply \cite[Proposition 2.6 and Theorem 2.5]{bouchard2021c} to deduce \eqref{eq: Ito vr}. 
  \end{proof}


\begin{thebibliography}{10}

\bibitem{bouchard2021c}
Bruno Bouchard, Gr{\'e}goire Loeper, and Xiaolu Tan.
\newblock A $C^{0, 1}$ -functional It\^{o}'s formula and its applications in
  mathematical finance.
\newblock {\em Stochastic Processes and their Applications}, 148:1299--323,
  2022.

\bibitem{bouchard2021approximate}
Bruno Bouchard, Gr{\'e}goire Loeper, and Xiaolu Tan.
\newblock Approximate viscosity solutions of path-dependent pdes and dupire's
  vertical differentiability.
\newblock {\em Annals of Applied Probability}, to appear.

\bibitem{cont2013functional}
Rama Cont and David-Antoine Fourni{\'e}.
\newblock Functional It{\^o} calculus and stochastic integral representation of
  martingales.
\newblock {\em The Annals of Probability}, 41(1):109--133, 2013.

\bibitem{cosso2021path}
Andrea Cosso, Fausto Gozzi, Mauro Rosestolato, and Francesco Russo.
\newblock Path-dependent Hamilton-Jacobi-Bellman equation: Uniqueness of
  Crandall-Lions viscosity solutions.
\newblock {\em arXiv preprint arXiv:2107.05959}, 2021.

\bibitem{delarue2010density}
Fran{\c{c}}ois Delarue and St{\'e}phane Menozzi.
\newblock Density estimates for a random noise propagating through a chain of
  differential equations.
\newblock {\em Journal of functional analysis}, 259(6):1577--1630, 2010.

\bibitem{dupireito}
Bruno Dupire.
\newblock Functional It\^{o} calculus.
\newblock {\em Portfolio Research Paper}, 04, 2009.

\bibitem{ekren2014viscosity}
Ibrahim Ekren, Christian Keller, Nizar Touzi, and Jianfeng Zhang.
\newblock On viscosity solutions of path dependent pdes.
\newblock {\em The Annals of Probability}, 42(1):204--236, 2014.

\bibitem{francesco2005class}
Marco~Di Francesco and Andrea Pascucci.
\newblock On a class of degenerate parabolic equations of Kolmogorov type.
\newblock {\em Applied Mathematics Research eXpress}, 2005(3):77--116, 2005.

\bibitem{friedman2008partial}
Avner Friedman.
\newblock {\em Partial differential equations of parabolic type}.
\newblock Courier Dover Publications, 2008.

\bibitem{kolmogorov1934theorie}
Andre\"i N.~Kolmogorov.
\newblock  Zufällige Bewegungen (zur Theorie der Brownschen Bewegung).
\newblock {\em Ann Math}, 35:116--117, 1934.

\bibitem{lanconelli2002linear}
Ermanno Lanconelli, Andrea Pascucci, and Sergio Polidoro.
\newblock Linear and nonlinear ultraparabolic equations of kolmogorov type
  arising in diffusion theory and in finance.
\newblock {\em Nonlinear problems in mathematical physics and related topics,
  II}, 2:243--265, 2002.

\bibitem{ren2014comparison}
Zhenjie Ren, Nizar Touzi, and Jianfeng Zhang.
\newblock Comparison of viscosity solutions of semi-linear path-dependent PDEs.
\newblock {\em SIAM Journal on Control and Optimization}, 58(1):277-302, 2020.


\bibitem{sonin1967class}
Isaac M.~Sonin.
\newblock On a class of degenerate diffusion processes.
\newblock {\em Theory of Probability \& Its Applications}, 12(3):490--496,
  1967.

\bibitem{stroock1997multidimensional}
Daniel~W.~Stroock and  Srinivasa R.~Varadhan.
\newblock {\em Multidimensional diffusion processes}, volume 233.
\newblock Springer Science \& Business Media, 1997.

\bibitem{weber1951fundamental}
Maria Weber.
\newblock The fundamental solution of a degenerate partial differential
  equation of parabolic type.
\newblock {\em Transactions of the American Mathematical Society},
  71(1):24--37, 1951.

\bibitem{zhou2020viscosity}
Jianjun Zhou.
\newblock Viscosity solutions to second order path-dependent
  hamilton-jacobi-bellman equations and applications.
\newblock {\em arXiv preprint arXiv:2005.05309}, 2020.

\end{thebibliography}

\def\cprime{$'$} \def\cprime{$'$}

\end{document}